\documentclass{amsart}
\usepackage{a4, amsmath, bm, amssymb}
\usepackage{setspace}
\usepackage{subfiles}
\usepackage{cite}
\usepackage{enumerate}
\usepackage{url}

\usepackage{faktor}
\usepackage{dsfont}
\usepackage{amsthm}

\makeatletter
\newtheorem*{rep@theorem}{\rep@title}
\newcommand{\newreptheorem}[2]{%
\newenvironment{rep#1}[1]{%
 \def\rep@title{#2 \ref{##1}}%
 \begin{rep@theorem}}%
 {\end{rep@theorem}}}
\makeatother

\newtheorem{theorem}{Theorem}

\newreptheorem{theorem}{Theorem}
\newtheorem{lemma}{Lemma}[section]
\newreptheorem{lemma}{Lemma}
\newtheorem{claim}{Claim}[section]
\newreptheorem{claim}{Claim}
\newtheorem{proposition}{Proposition}[section]
\newtheorem{corollary}{Corollary}[section]
 
\theoremstyle{definition}
\newtheorem{definition}{Definition}[section]

\newtheorem{remark}{Remark}[section]

\begin{document}

\title[Critical exponents and Kazhdan distance]{Critical exponents of normal subgroups, the spectrum of group extended transfer operators, and Kazhdan distance }

\author{Rhiannon \textsc{Dougall}}

\address{Mathematics Institute, University of Warwick,
Coventry CV4 7AL, U.K.}
\email{R.Dougall@warwick.ac.uk}

\begin{abstract}
For a pinched Hadamard manifold $X$ and a discrete group of isometries $\Gamma$ of $X$, the critical exponent $\delta_\Gamma$ is the exponential growth rate of the orbit of a point in $X$ under the action of $\Gamma$. We show that the critical exponent for any family $\mathcal{N}$ of normal subgroups of $\Gamma_0$ has the same coarse behaviour as the Kazhdan distances for the right regular representations of the quotients $\Gamma_0/\Gamma$. The key tool is to analyse the spectrum of transfer operators associated to subshifts of finite type, for which we obtain a result of independent interest. 
\end{abstract}

\maketitle
\section{Introduction}

Let $X$ be a simply connected, complete Riemannian manifold whose curvatures are bounded between two negative constants -- this is sometimes called a \emph{pinched Hadamard manifold}. For any non-elementary discrete group of isometries $\Gamma$ of $X$, the $\Gamma$-orbit of a point inside a ball of radius $R$ grows exponentially $R$. More precisely, define the \emph{critical exponent} $\delta_\Gamma$ by
\[
\delta_\Gamma=\limsup_{R\to \infty} \frac{1}{R}\#\log \left\{ g\in \Gamma : d(x,g x)\le R \right\} .
\]
It is easy to see that the definition is independent of $x\in X$. When $\Gamma$ is non-elementary, the limit exists and $\delta_\Gamma>0$ -- see, for instance, \cite{PPS}. 

If $\Gamma$ is torsion-free then we may form the quotient manifold $M=X/\Gamma$ and geodesic flow $\phi^t:SM\to SM$ on the unit tangent bundle $SM$. We refer to a closed geodesic $\gamma$ in $M$ and the corresponding periodic orbit $\gamma$ in $SM$ interchangeably. Write $\text{Per}(\phi)$ for the collection of periodic orbits, and write $|\gamma|$ for the length of a geodesic and the period of the orbit. We say that a point $x\in SM$ is \emph{wandering} if it is contained in a neighbourhood $U$ such that $\phi^t U \cap U = \varnothing$ for all large $t$. The \emph{non-wandering} set $\Omega(\phi)$ is the collection of points that are not wandering. Note that when $M$ is compact, $\Omega(\phi)=SM$. Following \cite{PPS}, the \emph{Gurevi\v{c} pressure}, $\mathcal{P}(\phi)$, of the geodesic flow is defined by
\[
\mathcal{P}(\phi) = \limsup_{T\to\infty}\frac{1}{T}\#\left\{\gamma\in \text{Per}(\phi) : |\gamma|\le T, \gamma\cap W \ne \varnothing \right\},
\]
where $W$ is any (non-empty) open subset of $\Omega$ with compact closure. Note that \cite{PPS} defines the Gurevi\v{c} pressure $\mathcal{P}(\phi,F)$ more generally for a potential $F$, but in our case $F=0$ and so we simplify the notation $\mathcal{P}(\phi,0)=\mathcal{P}(\phi)$. Also contained in \cite{PPS} is the proof that $\mathcal{P}(\phi)=\delta_\Gamma$ when $\mathcal{P}(\phi)>0$ (which is satisfied in our later specialisation).

If $\Gamma$ is cocompact, then
\[
\delta_\Gamma = \lim_{R\to\infty}\frac{1}{R}\log \mathrm{Vol}(x,R),
\]
where $\mathrm{Vol}(x,R)$ is the volume of an $R$-ball around $x\in X$. For $n$-dimensional quaternionic hyperbolic space $\bm{H}^n_{\mathbb{H}}$, with $n\ge 2$, and the Cayley plane $\bm{H}^2_{\mathbb{O}}$, Corlette \cite{Corlette} showed that this value is isolated in the following way. If $\Gamma$ is a lattice in $\mathrm{Isom}^+(\bm{H}^n_{\mathbb{H}})$, then $\delta_\Gamma = 4n+2$; and otherwise $\delta_\Gamma \le 4n$. There is also the corresponding statement for the Cayley plane: $\delta_\Gamma = 22$ when $\Gamma$ is a lattice in $\mathrm{Isom}^+(\bm{H}^2_{\mathbb{O}})$; and otherwise $\delta_\Gamma \le 16$. Therefore, in each case, for a fixed lattice $\Gamma_0$, we have that $\delta_\Gamma\le \delta_{\Gamma_0}-2$ for all infinite index $\Gamma\le \Gamma_0$. The symmetry of the spaces $\bm{H}^n_{\mathbb{H}}$ and $\bm{H}^2_{\mathbb{O}}$ are notable in the approach to this problem, which we discuss in the next section. 

In this paper we develop a dynamical approach to analyse the critical exponent of normal subgroups $\Gamma\unlhd \Gamma_0$, of a fixed (torsion-free) \emph{convex cocompact} $\Gamma_0$. The \emph{convex cocompact} hypothesis says that the geodesic flow $\phi_0^t:SM_0\to SM_0$, where $M_0=X/\Gamma_0$, has compact non-wandering set $\Omega(\phi_0)$. 

For any $\Gamma \unlhd \Gamma_0$, we have $\delta_\Gamma\le \delta_{\Gamma_0}$ and moreover $\delta_{\Gamma}=\delta_{\Gamma_0}$ precisely when $\Gamma_0/\Gamma$ is amenable (we discuss the history of this result later). Consequently $\delta_{\Gamma}<\delta_{\Gamma_0}$ when $\Gamma_0/\Gamma$ is non-amenable. 
Our result is to describe coarse behaviour of $\delta_\Gamma$, over any family $\mathcal{N}$ of normal subgroups of $\Gamma_0$, in terms of \emph{Kazhdan distances} associated to the quotients $\Gamma_0/\Gamma$, which we explain below. Some natural families of coverings are a tower of regular covers $M_1 \to M_2 \to \cdots \to M_0$, corresponding to a family  $\Gamma_1 \le \Gamma_2 \le \cdots \le \Gamma_0$ of normal subgroups of $\Gamma_0$; and the family of all non-amenable regular covers of $M_0$, i.e. all $\Gamma\unlhd \Gamma_0$ for which $\Gamma_0/\Gamma$ is non-amenable.

In the following, $G$ is assumed to be a countable group (however, many of the definitions can be made in the setting of locally compact groups). We present the definition of an amenable group $G$ due to F{\o}lner \cite{Folner}. 
A group $G$ is \emph{amenable} if for every $\epsilon>0$, and for every finite set $A$, there exists a set $E$ which is $\epsilon,A-$invariant; that is,
\[
\#E\Delta E a\le \epsilon \# E,
\]
for all $a\in A$.

Write $\mathds{1}_E\in \ell^2(G)$ for the indicator function on the set $E$. Noting that $\#E\Delta E a = \left|\mathds{1}_E-\mathds{1}_{E a}\right|$, there is the following equivalent definition
in terms of the right regular representation $\pi_{G}: G \to \mathcal{U}(\ell^2(G))$, $(\pi_{G}(g)f)(x)=f(xg)$, due to Hulanicki \cite{Hulanicki}. A group $G$ is \emph{amenable} if and only if, for any finite generating set $A\subset G$, 
\[
\inf_{v\in \ell^2(G), \left|v\right|=1} \max_{a\in A} \left|\pi_{G}(a) v-v\right|=0.
\]
For any unitary representation $\rho : G \to \mathcal{U}(\mathcal{H})$ in a Hilbert space $(\mathcal{H},\left|\cdot\right|)$, the quantity $\kappa_A(\rho,\mathds{1})$ defined by
\[
\kappa_A(\rho,\mathds{1}):=\inf_{v\in V, \left|v\right|=1} \max_{a\in A} \left|\rho(a) v-v\right|
\]
is called the \emph{Kazhdan distance (between $\rho$ and the trivial representation $\mathds{1}$)}.
A group is said to have \emph{property} (T) if there is some $\kappa>0$ such that $\kappa_A(\rho,\mathds{1})>\kappa$ for all unitary representations that have no invariant vector. 

The main theorem of this paper is the following.
\begin{theorem}\label{manifold}
Let $\Gamma_0$ be a convex cocompact group of isometries of a pinched Hadamard manifold $X$, and let $A$ be a finite generating set for $\Gamma_0$. For any collection $\mathcal{N}$ of normal subgroups of $\Gamma_0$, we have
\[
\sup_{\Gamma \in \mathcal{N}} \delta_{\Gamma}<\delta_{\Gamma_0} \, \text{ if and only if } \, \inf_{\Gamma \in \mathcal{N}} \kappa_{A/\Gamma}(\pi_{\Gamma_0/\Gamma},\mathds{1}) > 0.
\]
\end{theorem}
We remark that this theorem is reminiscent of results on the bottom of the spectrum of the Laplacian (for example by Sunada \cite{Sunada}) which we discuss in the next section.

If $\Gamma_0$ has property (T), then we have that 
\[
\inf_{\Gamma \unlhd \Gamma_0 : [\Gamma_0:\Gamma]=\infty} \kappa_{A/\Gamma}(\pi_{\Gamma_0/\Gamma},\mathds{1}) > 0.
\]
We remark that, in this case, $[\Gamma_0:\Gamma]=\infty$ is equivalent to $\Gamma_0/\Gamma$ being non-amenable.

It is known that the isometry group of for real hyperbolic space $\bm{H}^n_{\mathbb{R}}$ does not have property (T), and so its cocompact subgroups also do not satisfy property (T). However, when we consider groups arising from variable curvature, we do find cocompact examples with property (T). Indeed, the mechanism behind the gap in the critical exponents for $\bm{H}^n_{\mathbb{H}}$ and $\bm{H}^2_{\mathbb{O}}$, as shown by Corlette, is the fact their isometry groups have property (T) \cite{Bekka}.

\begin{corollary}
With the hypotheses of Theorem \ref{manifold}, if $\Gamma_0$ has property (T) then 
\[
\sup\left\{ \delta_\Gamma : \Gamma\unlhd \Gamma_0, [\Gamma_0:\Gamma]=\infty\right\} < \delta_{\Gamma_0}.
\]
\end{corollary}

The proof of Theorem \ref{manifold} relies on an analysis of the dynamics of the geodesic flow, and in particular the symbolic dynamics for the geodesic flow. In this way, we relate the problem to the spectrum of group extended transfer operators. We prove an analogous theorem about the spectrum of group extended transfer operators which is of independent interest. This approach is a departure from the methods employed for the symmetric spaces.

This work was carried out as part of the author's PhD thesis. The author is therefore very grateful to Richard Sharp for introducing her to the topic, and for his careful reading of this manuscript.

\section{History and background}

A classical example of the interplay between combinatorial properties of a group, and the geometry on which it acts, is given in Brooks \cite{Brooks2}, \cite{Brooks}. Let $M \to M_0$ be a regular covering of a Riemannian manifold $M_0$ of ``finite topological type" (i.e. $M_0$ is the union of finitely many simplices). Let $\lambda_0(M)$ and $\lambda_0(M_0)$ denote the bottom of the spectrum of the Laplacian on $M$ and $M_0$ respectively. Brooks shows that $\lambda_0(M)=\lambda_0(M_0)$ if and only if the group of deck transformation given by the covering is amenable. We will refer to a result of this form as \emph{an amenability dichotomy}.
This was extended by Sunada \cite{Sunada} in the following way. We assume now that $M_0$ is compact, and so $\lambda_0(M_0)=0$, and write $M_0=X/\Gamma_0$. For any $\Gamma\unlhd \Gamma_0$ we get a regular cover $M_\Gamma=X/\Gamma$ of $M_0$. Sunada shows that for any finite generating set $A$ of $\Gamma_0$, there are constants $c_1,c_2$ depending only on the geometry of $X$ and on $A$, such that for any regular cover $M_\Gamma=X/\Gamma$, we have
\[
c_1 (\kappa_{A/\Gamma}(\pi_{\Gamma_0/\Gamma},\mathds{1}))^2 \le \lambda_0(M) \le c_2 (\kappa_{A/\Gamma}(\pi_{\Gamma_0/\Gamma},\mathds{1}))^2,
\]
where $A/\Gamma$ denotes the projection of $A$ to the quotient $\Gamma_0/\Gamma$.
These results were also generalised by Roblin and Tapie \cite{RoblinTapie}, and in the thesis of Tapie, relating the difference $\lambda_0(M_0)-\lambda_0(M)$ to the bottom of the spectrum of a combinatorial Laplacian (which is in turn related to the Kazhdan distance).

A consequence of these spectral results is that, for any family of normal subgroups $\mathcal{N}$ of $\Gamma_0$, we have 
\[
\inf_{\Gamma\in \mathcal{N}}\lambda_0(X/\Gamma)= 0 \text{ if and only if }\inf_{\Gamma\in \mathcal{N}}\kappa_{A/\Gamma}(\pi_{\Gamma_0/\Gamma},\mathds{1})= 0.
\]

For real hyperbolic space $\bm{H}^n_{\mathbb{R}}$, the spectral geometry and dynamics are related by the celebrated Patterson-Sullivan theorem \cite{Sullivan}. We have that,
\[
\lambda_0(X/\Gamma) = \left\{ \begin{array}{cc}
\delta_\Gamma(n-\delta_\Gamma) & \text{ if } \delta_\Gamma \ge \frac{n}{2}
\\
\frac{n^2}{4} & \text{ if } \delta_\Gamma \le \frac{n}{2}.
\end{array}
\right.
\]
There are analogous statements for the other noncompact rank 1 symmetric spaces. However these results fail to extend to spaces which do not satisfy such strong symmetry hypotheses. 

We now return to the setting of the introduction: $X$ is a pinched Hadamard manifold and $\Gamma_0$ is a (torsion-free) convex cocompact group of isometries. In this context, various authors have developed more dynamical methods to obtain an analogue of the amenability dichotomy of Brooks. With the hypotheses we have given $X$ and $\Gamma_0$, it was first showed by Roblin \cite{Roblin} that if $\Gamma_0/\Gamma$ is amenable, then $\delta_{\Gamma_0}=\delta_\Gamma$; and recently Dougall and Sharp \cite{DougallSharp} have shown the converse, that $\delta_{\Gamma_0}=\delta_\Gamma$ implies that $\Gamma_0/\Gamma$ is amenable. The difference in techniques is notable: Roblin's result was obtained by analysing Patterson-Sullivan measures on the boundary, whereas Dougall and Sharp exploit the symbolic dynamics for the geodesic flow -- it is this latter approach that we extend. Another important reference (that is key to \cite{DougallSharp}) is that of Stadlbauer \cite{Stadlbauer}, who obtained the equivalence in the setting of $X=\bm{H}^n_{\mathbb{R}}$ and $\Gamma_0$ essentially free, and whose techniques we discuss later.

\section{Subshifts of finite type, transfer operators and group extensions}\label{ssft}

For a finite alphabet $\mathcal{W}=\left\{ 1,\ldots , k \right\}$, we can give rules governing when two letters in the alphabet can be concatenated in terms of a $k\times k$ matrix $A$ with entries $0$ or $1$. Namely, for $i,j\in\mathcal{W}$, the concatenation $ij$ is said to be \emph{admissible} if $A(i,j)=1$. In this way, the set of admissible words of length $n$, $\mathcal{W}^n$, is the collection of concatenations $w=x_0\cdots x_{n-1}$, where $x_0,\ldots ,x_{n-1}\in \mathcal{W}$ and $A(x_i,x_{i+1})=1$ for all $i=0,\cdots, n-2$. Extending this to one-sided infinite words, define the (one-sided) shift space to be
\[
\Sigma^+ = \left\{ x_0x_1\cdots  \in \mathcal{W}^{\mathbb{Z}^+} : \forall i\in \mathbb{Z}^+,\, A(x_i,x_{i+1})=1 \right\} .
\]
The two-sided shift space $\Sigma$ is defined analogously by
\[
\Sigma=
\left\{ \cdots x_{-1}x_0x_1\cdots  \in \mathcal{W}^{\mathbb{Z}} : \forall i\in \mathbb{Z},\, A(x_i,x_{i+1})=1 \right\} .
\]
For brevity, we make the following definitions for the two-sided space $\Sigma$. However, they pass to $\Sigma^+$ by the canonical projection $\Sigma \to \Sigma^+$, given by forgetting past (negative) coordinates.

Write $x$ to denote an element 
 of $\Sigma$, and write $x^i$ for the sequence element at index $i$; in this way $x=(x^i)_{i\in \mathbb{Z}}$. Similarly, for an element $w\in\mathcal{W}^n$, we write $w^i$ to denote the $i$th element in the concatenation.
There is a natural dynamical system, $\sigma : \Sigma\to \Sigma$ 
called the \emph{shift map}, with defining property $\sigma(x)^i = x^{i+1}$. Together, we call the pair $(\Sigma,\sigma)$ a \emph{subshift of finite type}, or in some literature, a \emph{topological Markov chain}.

There is a natural topology with basis consisting of cylinder sets 
\[
[w]_j^k:= \left\{ x\in \Sigma : \forall j \le i \le k, \, x^i=w^{i-j}\right\},
\]
for any $j,k\in\mathbb{N}$, and $w\in\mathcal{W}^{k-j+1}$. For the one-sided shift space $\Sigma^+$, we often write $[w]=[w]_0^{n-1}$, where $w\in \mathcal{W}^n$. The topology is metrizable: for any $0<\theta<1$ define the metric $d_\theta$ by
\[
d_\theta(x,y) = \theta^{\inf\left\{|i|\,:\, x^i\ne y^i\right\}}
\]
and $d_\theta(x,x)=0$.

We always assume that $A$ is \emph{aperiodic}, that is to say there is some $N>0$ for which $A^N(i,j)>0$ for all $i,j\in\left\{1,\ldots k\right\}$. We call the minimal $N$ the \emph{aperiodicity constant}. The assumption that $A$ is aperiodic is equivalent to  $\sigma$ being \emph{topologically mixing}, i.e. for any non-empty open sets $U,V\subset \Sigma$ there is $N$ for which $\sigma^{-n}(U)\cap V\ne\varnothing$ for all $n\ge N$.

The \emph{H\"{o}lder continuous functions} $f: \Sigma \to \mathbb{R}$ are defined by the existence of $0<\alpha\le 1$ (a \emph{H\"{o}lder exponent}) and $C$ such that for any $x,y\in \Sigma$, 
\[
|f(x)-f(y)|\le Cd_\theta(x,y)^\alpha.
\]
Notice that replacing $\theta$ by $\theta^{1/\alpha}$ gives that $f$ has H\"{o}lder exponent $1$ in this new metric. We will later fix a H\"{o}lder continuous $r$, and then assume that the metric is chosen to give $r$ H\"{o}lder exponent $1$. 

For a strictly positive H\"{o}lder continuous function $r:\Sigma\to \mathbb{R}$, define the \emph{suspension space} $\Sigma_r$ (with \emph{suspension} $r$) by
\[
\Sigma_r = \Sigma\times \mathbb{R} / \sim_r,
\]
where $\sim_r$ is the equivalence relation $(x,s) \sim_r (\sigma x, s-r(x))$. We define the suspension flow $\sigma^t_r:\Sigma_r \to \Sigma_r$, locally by $\sigma^{t}_r(x,s)= (x, s+t)$.
The \emph{pressure} $P(f,\sigma)$ of a H\"{o}lder continuous $f: \Sigma\to \mathbb{R}$
 is defined to be
\[
P(f,\sigma)= \lim_{n\to\infty} \frac{1}{n}\log\sum_{\substack{x\in \Sigma :\\ \sigma ^n x=x}} e^{f^n(x)}.
\]

We now specialise to the one-sided shift space $\Sigma^+$.
For a H\"{o}lder continuous $r:\Sigma^+\to \mathbb{R}$, we define the \emph{transfer operator} $L_r: C(\Sigma^+,\mathbb{R})\to C(\Sigma^+,\mathbb{R})$ by
\[
L_r f(x) = \sum_{\substack{y\in \Sigma^+ :\\ \sigma y = x}} e^{r(y)}f(y),
\]
where $C(\Sigma^+,\mathbb{R})$ is the Banach space of continuous functions with the supremum norm $\|\cdot\|_\infty$. We write $\text{spr}(L_r)$ for the spectral radius of $L_r$ in this Banach space, omitting explicit reference to the space. We find better spectral properties for $L_r$ when we restrict to the smaller Banach space of H\"{o}lder continuous functions. Write $F^+_\theta$ for the functions $f:\Sigma^+\to\mathbb{R}$ that are Lipschitz (have H\"{o}lder exponent $1$) in the $d_\theta$ metric. Define the semi-norm
\[
|f|_\theta=\sup_{n\in\mathbb{N}}\sup_{\substack{x,y :\\  x^i=y^i,\, |i|\le n}}\frac{|f(x)-f(y)|}{\theta^n}.
\]
Then $(F_\theta,\|\cdot\|_\theta)$ is a Banach space with the norm $\|\cdot\|_\theta = \|\cdot \|_\infty + |\cdot |_\theta$. By the Ruelle Perron Frobenious theorem \cite{PP}, $L_r$ has a simple, isolated, maximal eigenvalue at $e^{P(r,\sigma)}$. 
 
For a countable group $G$, define the \emph{group extension (with skewing function $\psi:\Sigma^+\to G$)}, $T=T_\psi : \Sigma^+\times G\to \Sigma^+\times G$, to be the product space $\Sigma^+\times G$ together with dynamical system 
\[
T(x,g) = (\sigma x,g(\psi(x))^{-1})
\]
(Note that in \cite{DougallSharp} the group extension by $\psi$ is defined to be $T(x,g)=(x,g\psi(x))$, and so to translate to the present terminology we need to take the inverse $\psi^{-1}$. We have chosen this convention so as to simplify the expression for the transfer operator.) In the following $\psi$ is always assumed to depend only on one letter. In this way, we can think of $\psi$ as a function $\psi: \mathcal{W} \to G$. Moreover, for every $n\in\mathbb{N}$, we define $\psi^n : \Sigma^+ \to G$, $\psi^n(x)=\psi(x^0)\cdots \psi(x^{n-1})$, and write $\psi^n: \mathcal{W}^n\to G$, $\psi^n(w)=\psi(w^0)\cdots \psi(w^{n-1})$. 

For $r:\Sigma^+\to \mathbb{R}$, there is a unique $\widetilde{r}:\Sigma^+\times G\to \mathbb{R}$ such that $\widetilde{r}(x,g)=r(x)$. We therefore dispense with the cumbersome tilde, and simply write this function as $r:\Sigma^+\times G\to \mathbb{R}$. 
Define the \emph{group extended transfer operator} $\mathcal{L}_r$ pointwise by 
\[
\mathcal{L}_{r} f(x,g) = \sum_{\substack{(y,g^*)\in \Sigma^+\times G :\\  T (y,g^*)=(x,g)}} e^{r(y)}f(y,g^*) = \sum_{\substack{y\in \Sigma^+ :\\ \sigma y = x}} e^{r(y)}f(y,g\psi(y)),
\] 
and so
\[
\mathcal{L}^n_{r} f(x,g) = \sum_{\substack{(y,g^*)\in \Sigma^+\times G :\\  T^n (y,g^*)=(x,g)}} e^{r^n(y)}f(y,g^*) = \sum_{\substack{y\in \Sigma^+ :\\ \sigma^n y = x}} e^{r(y)}f(y,(g\psi(\sigma^{n-1}y)\cdots\psi(y)).
\]

Define the Banach space $(\mathcal{C}^\infty, \|\cdot \|)$ by
\[
\mathcal{C}^\infty = \left\{ f\in C(\Sigma^+\times G, \mathbb{R}) : \|f\|<\infty \right\} ,
\]
\[
\|f\| = \sqrt{\sum_{g\in G} \sup_{x\in\Sigma^+} |f(x,g)|^2 }.
\]
Then $\mathcal{L}_{r} : \mathcal{C}^\infty \to \mathcal{C}^\infty$ is a bounded operator. 
We will always take the spectral radius $\text{spr}(\mathcal{L}_r)$ with respect to the $\mathcal{C}^\infty$ norm, and continue to omit reference to the space.

Following Sarig, \cite{Sarig}, define the \emph{Gurevi\v{c} pressure}  $P_{\mathrm{Gur}}(f,T_{\psi})$ of a H\"{o}lder continuous $f: \Sigma^+\times G \to \mathbb{R}$ by
\[
P_{\mathrm{Gur}}(f,T_{\psi})= \limsup_{n\to\infty} \frac{1}{n}\log\sum_{\substack{\sigma ^n x=x : \\ \psi^n(x)=e_G}} e^{f^n(x)},
\]
where $e_G$ is them identity of $G$.

In general we have that if $T_\psi$ is topologically transitive then $P_{\mathrm{Gur}}(r,T_{\psi})\le\log\text{spr}(\mathcal{L}_r)\le  \log\text{spr}(L_r)=P(r,\sigma)$. If we assume in addition that the pair $(\psi,r)$ is \emph{weakly symmetric}, in the following sense, then we have that $ P_{\mathrm{Gur}}(r,T_{\psi})=\log\text{spr}(\mathcal{L}_r)$ \cite{Jaerisch}. We say an involution $\dag$ on $\mathcal{W}$ is \emph{weakly symmetric with respect to $r$} if there are real numbers $D_n$ such that $D_n^{1/n}\to 1$ and
\[
\sup_{x\in [w],y \in[w^\dag]}\frac{e^{r^n(x)}}{e^{r^n(y)}}\le D_n
\]
for every $w\in \mathcal{W}^n$, where $w^\dag:= (w^{n-1})^\dag\cdots (w^0)^\dag$.
We then say that the pair $(\psi,r)$ is \emph{weakly symmetric} if there is an involution $\dag$ such that $\psi(v^\dag) = \psi(v)^{-1}$ for all $v\in \mathcal{W}$, and such that $\dag$ is weakly symmetric with respect to $r$.

It will be useful to extend the definition of weak symmetry to $\Sigma$. We say that $\dag$ is symmetric with respect to $r: \Sigma\to \mathbb{R}$ if there is some $\beta(n)$ with $\beta(n)/n\to 0$ as $n\to \infty$, such that $|r^n(x)-r^n(x^\dag)|\le \beta(n)$ for any $x\in \Sigma$ with $\sigma^nx=x$, and $x^\dag\in\Sigma$ defined by $(x^\dag)^i = (x^{-i})^\dag$. In section \ref{geometry} we show that in the one-sided case $\Sigma^+$, the two definitions are equivalent (once we interperet $x^\dag\in\Sigma^+$ as $(x^\dag)^i = (x^{kn-i})^\dag$, for $i=0,\ldots, n-1$ and $k\in\mathbb{N}$).

\begin{proposition}[Stadlbauer\cite{Stadlbauer}\label{ThmStadl}]
Assume $T_\psi: \Sigma^+\times G \to \Sigma^+\times G$ is transitive. If $G$ is non-amenable, then $\mathrm{spr}(\mathcal{L}_r)<\mathrm{spr}(L_r)$. Assuming that $(\psi,r)$ is weakly symmetric, the converse holds: if $G$ is amenable, then $P_{\mathrm{Gur}}(r,T_{\psi})=P(r,\sigma)$
\end{proposition}

\begin{remark}
The statement of Stadlbauer's theorem is actually for the Banach space
\[
\mathcal{H}^\infty = \left\{ f: \Sigma^+\times G\to \mathbb{R} : f(\cdot,g)\in L^1(\Sigma^+,\mu_r) \text{ for all } g\in G, \|f\|_{\mathcal{H}^\infty}<\infty \right\},
\]
\[
\|f\|_{\mathcal{H}^\infty} = \sqrt{\sum_{g\in G} \left(\int |f(x,g)|d\mu_r\right)^2 }.
\]
where $\mu_r$ is the equilibrium state for $r$.
However, the proof only uses the fact that the spectrum is attained on the subset
\[
\left\{ f\in \mathcal{H}^\infty: f(x,g)=f(y,g) \text{ for all } x,y\in\Sigma^+, g\in G\right\}
\]
which is isometrically isomorphic to $\ell^2(G)$,
and so the statment is true for any Banach space with this property (we will show that $\mathcal{C}^\infty$ has this property later).
Moreover, Stadlbauer considers countable alphabets under a certain finiteness condition, namely having \emph{big images and pre-images}. As the application we have in mind is for a finite alphabet we limit ourselves to the finite alphabet case.
\end{remark}
\begin{remark}
Jaerisch \cite{Jaerisch} shows that we have the conclusion $\mathrm{spr}(\mathcal{L}_r)=\mathrm{spr}(L_r)$ if $G$ is amenable, under no symmetry hypothesis.
\end{remark}

In this paper we study a family of group extensions which can be seen as quotients of a fixed group extension $\psi: \Sigma^+\to G$. For each $H\unlhd G$, write $\psi_H(x)$ for the coset of $G/H$ given by $\psi(x)$. 
In this way $T_{\psi_{H}} : \Sigma^+\times G/H \to \Sigma^+\times G/H$ is a group extension with skewing function $\psi_H$.
For notational convenience, write 
$\mathcal{L}_{r,H}$ for the transfer operator given by $r$ and $T_{\psi_{H}}$; and write $\mathcal{C}_H^\infty$ for the Banach space assosciated to 
$\mathcal{L}_{r,H}$.
We also consider a family of transfer operators $\mathcal{L}_{r_s,H}$ where $s\mapsto r_s\in F_\theta$, $s\in[-1,1]$, is continuous in the $\|\cdot\|_\theta$ topology.

By Proposition \ref{ThmStadl}, if $H\unlhd G$ with $G/H$ non-amenable and $T_{\psi_{H}}: \Sigma^+\times G/H\to \Sigma^+\times G/H$ is transitive, then $\text{spr}(\mathcal{L}_{r,H})<\text{spr}(L_r)$. The proof in \cite{Stadlbauer} finds an upper bound for $\text{spr}(\mathcal{L}_{r,H})$ that depends on the first return to a cylinder under $T_{\psi_H}$. As this bound does not suffice for our needs, we introduce a new condition on $\psi$ that removes this dependency. 
\begin{definition}
We say that $(\Sigma^+,G,\psi)$ satisfies \emph{linear visibility with remainder $\mathrm{(LVR)}$} if
there exists a map $\chi: G \to \bigcup_{n=1}^\infty \mathcal{W}^n$ with the following properties:  
\begin{itemize}
\item
(visibility with remainder) there exists a finite set $\mathcal{R}\subset G$ such that for every $g\in G$, there are $r_1,r_2\in \mathcal{R}$ with $\psi^{k_g}(\chi(g))=r_1gr_2$, where $k_g$ is the length of the word $\chi(g)$; 
\item
(linear growth) there exists $L$ such that for any finite collection $g_1,\ldots, g_r\in G$, writing $g=g_1\cdots g_r$, we have that $k_g\le L(\sum_{i=1}^r k_{g_i})$, where $k_g$ is the length of $\chi(g)$, and $k_{g_i}$ the length of $\chi(g_i)$, for each $i$.
\end{itemize}
\end{definition}

\section{Statement of Main Results}
We restate our theorem about the behaviour of the critical exponent.
\begin{reptheorem}{manifold}
Let $\Gamma_0$ be a convex cocompact group of isometries of a pinched Hadamard manifold $X$, and let $A$ be a finite generating set for $\Gamma_0$. For any collection $\mathcal{N}$ of normal subgroups of $\Gamma_0$, we have
\[
\sup_{\Gamma \in \mathcal{N}} \delta_{\Gamma}<\delta_{\Gamma_0} \, \text{ if and only if } \, \inf_{\Gamma \in \mathcal{N}} \kappa_{A/\Gamma}(\pi_{\Gamma_0/\Gamma},\mathds{1}) > 0.
\]
\end{reptheorem}

As in \cite{DougallSharp}, the proof of the theorem uses the dynamics of the geodesic flow $\phi^t_\Gamma :SM_\Gamma \to SM_\Gamma$, for $M_\Gamma=X/\Gamma$, and in turn the dynamics of group extended shift spaces. Therefore it will be crucial to prove the following theorem about the spectrum of group extended transfer operators. This extends the results of Stadlbauer \cite{Stadlbauer}. 

As in the introduction, fix $\sigma: \Sigma^+\to \Sigma^+$ a topologically mixing subshift of finite type, and $r: \Sigma^+ \to \mathbb{R}$ a potential. Fix $G$ a countable group and $\psi : \Sigma^+ \to G$ constant on cylinders of length 1.

\begin{theorem}\label{transferop}
Let $A$ be a finite generating set for $G$, and let $\mathcal{N}$ be a collection of normal subgroups of $G$.
\begin{enumerate}[(i.)]
\item\label{transferop1} Assume that $(\psi,r)$ is weakly symmetric. Then
\[
\inf_{H\in \mathcal{N}} \kappa_{A/H}(\pi_{G/H},\mathds{1}) = 0 \implies \sup_{H\in \mathcal{N}} P_{\mathrm{Gur}}(r,T_{\psi_{H}}) = P(r,\sigma).
\]
\item\label{transferop2} Assume that $(\Sigma^+,G,\psi)$ satisfies $\mathrm{(LVR)}$. Then
\[
\inf_{H \in \mathcal{N}}\kappa_{A/H}(\pi_{G/H},\mathds{1})>0
\implies
\sup_{H \in \mathcal{N}}
\mathrm{spr}(\mathcal{L}_{r,H}) < \mathrm{spr}(L_r).
\]
\item \label{transferop3} In addition, in case (\ref{transferop2}) suppose that $s\mapsto r_s$ is continuous (in the $\|\cdot\|_\theta$ topology) for $s\in [-1,1]$. 
Then
\[
\inf_{H \in \mathcal{N}}\kappa_{A/H}(\pi_{G/H},\mathds{1})>0
\implies
\sup_{H \in \mathcal{N}, s\in [-\delta,\delta]}
\mathrm{spr}(\mathcal{L}_{r_s,H}) < \mathrm{spr}(L_{r_0}),
\]
for some $\delta>0$.
\end{enumerate} 
\end{theorem}

\begin{remark}
If $\psi: \Sigma^+\to G$ depends on $n$-coordinates, as opposed to one, then we may still apply the conclusions of Theorem \ref{transferop}. To see this, let $\Sigma^+_n$ denote the subshift of finite type whose alphabet is given by admissible words of length $n$ for $\Sigma^+$, and with transition matrix $A_n(u,v)=1$ if and only if $u^{i+1}=v^i$ for all $i=0,\ldots, n-2$. Then $\psi$ gives rise to $\psi_n: \Sigma^+_n\to G$ depending only on one coordinate. Moreover, the statistics for the H\"{o}lder continuous functions, pressure and transfer operators pass to $(\Sigma^+_N,\psi_n)$ in the natural way. 
\end{remark}

\section{Axiom A flows and symbolic coding}
We now take a brief excursion into the theory of Smale's Axiom A flows \cite{Smale} and the symbolic coding of Bowen \cite{Bowen}.

Throughout, $f^t$ is a smooth flow on a complete Riemannian manifold $N$. 

A closed, $f^t$-invariant set $\Lambda\subset N$ is said to be \emph{hyperbolic} if 
there is a continuous, $Df^t$-invariant splitting of the tangent bundle
\[
T_{\Lambda}(N) = E^0\oplus E^s \oplus E^u
\]
and constants $\lambda,C>0$ such that
\begin{itemize}
\item
$E^0$ the line bundle tangent to the flow direction;
\item $\|Df^t v\| \le C e^{-\lambda t}\|v\|$ for all $v\in E^s$;
\item $\|Df^{-t} v\| \le C e^{-\lambda t}\|v\|$ for all $v\in E^u$;
\end{itemize}
We remark that this definition is independent of the choice of metric when $\Lambda$ is compact. 

The set $\Lambda$ is said to be a \emph{basic set} if
\begin{enumerate}
\item $\Lambda$ is compact and hyperbolic;
\item $f^t_\Lambda$ is transitive;
\item the periodic orbits for $f^t_\Lambda$ are dense in $\Lambda$; and
\item there is an open neighbourhood $U$ of $\Lambda$ such that $\bigcap_{t\in\mathbb{R}} f^t (U)=\Lambda$.
\end{enumerate}
The flow $f^t$ satisfies Smale's \emph{Axiom A} if the non-wandering set $\Omega(f)$ is a finite union of basic sets.

Let $\phi_0^t:SM_0\to SM_0$ be the geodesic flow given by $M_0=X/\Gamma_0$ where $\Gamma_0$ is convex cocompact. Then the non-wandering set $\Omega(\phi_0)$ is a basic set. (See for instance \cite[Chapter 17]{KH} in the case that $M$ is compact.)

We now describe some of the constructions relating to the theory of (hyperbolic) basic sets, which play an important role in Bowen's symbolic coding.

For $x\in \Omega(f)$ define the (strong) \emph{local stable manifold} $W_\epsilon^s(x)$ and (strong) \emph{local unstable manifold} $W_\epsilon^u(x)$ by
\[
W_\epsilon^s(x) = \left\{y\in N : d(f^t(x),f^t(y))\le \epsilon \text{ for all }t, \lim_{t\to \infty}d(f^t(x),f^t(y))=0 \right\}
\]
\[
W_\epsilon^u(x) = \left\{y\in N : d(f^{-t}(x),f^{-t}(y))\le \epsilon \text{ for all }t, \lim_{t\to \infty}d(f^{-t}(x),f^{-t}(y))=0 \right\}
\]
For small enough $\epsilon$, these sets are diffeomorphic to embedded disks of codimension $1$. These sets give us a \emph{local product structure} $[\cdot,\cdot]$. For sufficiently close $x,y$, we have that $W_\epsilon^s(x)\cap W_\epsilon(\phi_0^t(y))\ne\varnothing$ for a unique $t\in[-\epsilon,\epsilon]$, and we define $[x,y]$ to be this intersection point.

Suppose that $D^1,\ldots , D^k$ are codimension $1$ disks that form a local cross-section to the flow.
We say that $S^i\subset \mathrm{int}(D^i)\cap \Omega(f)$ is a \emph{rectangle} if $x,y\in S^i$ implies that $[x,y]= f^t z$, for some $z\in D^i$, $t\in [-\epsilon,\epsilon]$. We say that $S^i$ is \emph{proper} if $\overline{\mathrm{int}(S^i)}=S^i$, where the interior is taken relative to $D^i\cap\Omega(f)$.

Write $P$ for the Poincar\'{e} map $P: \bigcup_{i=1}^kS^i\to \bigcup_{i=1}^kS^i$. Write $W^s_\epsilon(x,S^i)$ and $W^u_\epsilon(x,S^i)$ for the projection of $W^s_\epsilon(x)$ and $W^u_\epsilon(x)$ onto $S^i$ respectively. We say that $\left\{S^1,\cdots, S^k\right\}$ is a \emph{Markov section} if
\begin{itemize}
\item $x\in \mathrm{int}(S^i)$ and $Px\in \mathrm{int}(S^j)$ implies that $P(W^s_\epsilon(x,S^i))\subset W^s_\epsilon(Px,S^j))$; and 
\item $x\in \mathrm{int}(S^i)$ and $P^{-1}x\in \mathrm{int}(S^j)$ implies that $P^{-1}(W^u_\epsilon(x,S^i))\subset W^s_\epsilon(P^{-1}x,S^j))$.
\end{itemize}

\begin{proposition}[Bowen \cite{Bowen}]
For all sufficiently small $\epsilon>0$, $f^t$ has a Markov section $\left\{S^1,\cdots, S^k\right\}$ such that $\mathrm{diam}(S^i)\le \epsilon$ for each $i$, and $\bigcup_{t\in [-\epsilon,\epsilon]}f^t(\cup_{i=1}^k S^i)= \Omega(f)$.
\end{proposition}

These Markov sections provide us with a `symbolic coding' for the geodesic flow. In the following, the Markov section $\left\{S^1,\cdots, S^k \right\}$ plays the role of an alphabet for a subshift of finite type $\Sigma$ with transition matrix $A$, defined by $A(i,j)=1$ if there is $x\in \mathrm{int}(S^i)$ with $Px\in \mathrm{int}(S^j)$. 

We specialise to the geodesic flow $\phi_0^t$ for the statement of the concluding proposition. We write $N_{\phi_0^t}(T)$ and $N_{\sigma_r^t}(T)$ for the number of periodic orbits of $\phi_0^t$ and $\sigma^t_r$, respectively, whose period is at most $T$. It is well known that $\delta_{\Gamma_0} = \lim_{T\to\infty} \frac{1}{T}\log N_{\phi_0^t}(T)$.

\begin{proposition}[Bowen \cite{Bowen}]\label{Bowen}
There is a mixing subshift of finite type $\sigma :\Sigma\to \Sigma$, a positive H\"{o}lder potential $r:\Sigma\to\mathbb{R}^+$ such that the suspended flow $\sigma^t_r: \Sigma^r\to \Sigma^r$ is semi-conjugate to  $\phi_0^t:\Omega(\phi_0)\to \Omega(\phi_0)$, i.e. there is a H\"{o}lder continuous $\theta: \Sigma^r\to \Omega(\phi_0)$ such that $\theta\circ \sigma_r^t = \phi_0^t\circ \theta$. Although $\theta$ is not a bijection, we have $N_{\sigma^t_r}(T)= N_{\phi_0^t}(T) + O(e^{h^\prime T})$, for some $h^\prime<\delta_{\Gamma_0}$.
\end{proposition}

We write $\left\{S^1_0,\ldots, S^k_0 \right\}$ for the Markov section for $\phi_0^t$. 
By Adachi \cite{Adachi} and Rees \cite{Rees} we may assume the sections have been chosen to reflect the involution $\iota: SM_0\to SM_0$, $\iota(x,v) = (x,-v)$. That is, we may assume that there is a fixed point free involution $\dag$ on $\left\{1,\cdots,k\right\}$ such that $\iota(S_0^i)=S_0^{i^\dag}$ for each $i$.


\section{Proof of Theorem \ref{manifold}}\label{geometry}
\subsection{Symbolic coding for the family of covers}
Write $p: SX \to SM_0$ for the induced covering map between tangent spaces, and $\pi: SX \to X$ for the usual projection. Fix $S^i\subset SX$ such that $p^{-1}: S^i_0\subset SM_0 \to S^i\subset SX$ is an isometry; and assume that these have been chosen with $\pi(S^i)=\pi(S^{i^\dag})$. Write $\mathcal{S} = \bigcup_{g\in \Gamma_0}\bigcup_{i=1}^kgS^i$, and $P$ for the Poincar\'{e} map $P: \mathcal{S}\to\mathcal{S}$.
Define $\psi: \Sigma^+ \to \Gamma_0$ by $\psi(x_0x_1) = g$ if there is $z\in \mathrm{int}(S^{x_0})$ such that $Pz\in \mathrm{int}(gS^{x_1})$. We will verify that $\psi$ is well-defined in the following lemma. For $\Gamma\unlhd \Gamma_0$, recall that we write $\psi_\Gamma$ for projection of $\psi$ to $\Gamma_0/\Gamma$.
Write $s=h_\Gamma$ for the (unique) value for which $P_{\mathrm{Gur}}(-sr,T_{\psi_\Gamma})=0$.

\begin{proposition}[Dougall-Sharp\cite{DougallSharp}]\label{DS}
The map $\psi$ is well-defined, and $(\psi,r)$ is weakly symmetric. When $\Gamma$ is non-trivial, $T_{\psi_\Gamma}$ is transitive.
If $h_\Gamma> h^\prime $ then $h_\Gamma=\delta_\Gamma$; otherwise $\delta_\Gamma\le h^\prime$, where $h^\prime$ is the constant from Proposition \ref{Bowen}. 
\end{proposition}

We give an indication of the proof.
\begin{proof}
For each $i,j$ with $A(i,j)=1$, fix a simply connected ball $V_0$ containing the pair $S_0^i,S_0^j$ (we may assume that $\epsilon$ was chosen sufficiently small to allow this). There is a unique connected component $V$ of $p^{-1}(V_0)$ containing $S^i$. Let $g\in \Gamma_0$ be the unique element for which $gS^j\subset V$. It follows that $\psi(ij)=g$; and therefore that $\psi$ is well-defined.

We show that $\dag$ is weakly symmetric with respect to $r$. Let $x\in \Sigma$ be a point of period $n$, and write $x^\dag\in\Sigma$ for the sequence defined by $(x^\dag)^{-i}=(x^i)^\dag$. Then $x$ and $x^\dag$ determine periodic points $z=\theta(x),z^\dag=\theta(x^\dag)\in SM_0$, which are related by $\iota(z)=z^\dag$. Therefore $z,z^\dag$ have identical period $T=r^n(x)=r^n(x^\dag)$, and so $|r^n(x)-r^n(x^\dag)|=0$ as required. Now, to complete the proof that $(\psi,r)$ is weakly symmetric, we observe that the symmetry of the rectangles $\iota(S_0^i)=S_0^{i^\dag}$ and the fact that we chose $S^i,S^{i^\dag}$ to satisfy $\pi(S^i)=\pi(S^{i^\dag})$ implies that $\psi(ij)=\psi(j^\dag i^\dag)^{-1}$.

Fix $\Gamma\unlhd \Gamma_0$. Recall that we write $M_\Gamma = X/\Gamma$ and $\phi^t_\Gamma: SM_\Gamma\to SM_\Gamma$ for the geodesic flow.
We give an equivalent definition of $\psi_\Gamma$.  Write $S^i_\Gamma\subset SM_\Gamma$ for the projection of $S^i$ to $SM_\Gamma$. Write $\mathcal{S}_\Gamma = \bigcup_{g\in \Gamma_0/\Gamma}\bigcup_{i=1}^k gS_\Gamma^i$, and $P_\Gamma$ for the Poincar\'{e} map $P_\Gamma: \mathcal{S}_\Gamma\to\mathcal{S}_\Gamma$. Then we have that $\psi_\Gamma(ij) = g\in \Gamma_0/\Gamma$ if and only if there is $z\in \mathrm{int}(S_\Gamma^i)$ such that $P_\Gamma z\in \mathrm{int}(g S_\Gamma^j)$.
The transitivity of $T_{\psi_\Gamma}$ for any non-trivial $\Gamma$ is therefore inherited by the transitivity of the geodesic flow.

Let $N_{\phi^t_0}(T,\Gamma)$ denote the number of orbits of $\phi_0^t$ whose period is at most $T$ and whose lift to $SM_\Gamma$ is closed. Analogously, let $N_{\sigma^t_r}(T,\Gamma)$ denote the periodic orbits of $\sigma^t_r$ whose period is at most $T$ and whose lift in the $r$-suspension over $T_{\psi_\Gamma}:\Sigma\times \Gamma_0/\Gamma\to \Sigma\times \Gamma_0/\Gamma$ is closed. Let $\gamma$ be a periodic $\phi_0^t$-orbit that passes through only the interior of rectangles. Then $\gamma$ has a closed lift in $SM_\Gamma$, if and only if its pre-image under the semi-conjugacy has a closed lift in the $r$-suspension over $T_{\psi_\Gamma}:\Sigma\times \Gamma_0/\Gamma\to \Sigma\times \Gamma_0/\Gamma$. (We use the equivalent definition of $\psi_\Gamma$ to observe this.)

In our setting, the ``$\limsup$" in the definition of the Gurevi\v{c} pressure is in fact a limit. Moreover, $\mathcal{P}(\phi_\Gamma)=\lim_{T\to\infty} \frac{1}{T}\log N_{\phi_0^t}(T,\Gamma)$. To see this, note that for $\mathrm{diam}(W)=\epsilon$, $g\gamma\cap W\ne\varnothing$ for at most $|\gamma|/\epsilon$ different $g\in \Gamma_0/\Gamma$. Now, note that $h_\Gamma=\limsup_{T\to\infty} \frac{1}{T}\log N_{\sigma^t_r}(T,\Gamma)$. The following inequalities are given by our equivalent definition of $\psi_\Gamma$, and the semi-conjugacy in Proposition \ref{Bowen},
\[
N_{\sigma^t_r}(T,\Gamma) - O(e^{h^\prime T})\le N_{\phi^t_0}(T,\Gamma) \le N_{\sigma^t_r}(T,\Gamma).
\]
Therefore, if $h_\Gamma> h^\prime$ then we must have that $h_\Gamma=\delta_\Gamma$. Since $\delta_\Gamma\le h_\Gamma$ we deduce that $h_\Gamma\le h^\prime$ implies that $\delta_\Gamma\le h^\prime$ too.
\end{proof}

We will use of the following constructions relating to the visual boundary $\partial X$ of $X$. We can give the unit tangent bundle $SX$ Hopf coordinates $(x,y,t)$, with respect to some fixed $o\in X$. (We suppress the dependency of the base-point $o$ as we only make statements up to reparamaterising the geodesic paths.) For any $x,y\in X$, denote by $[x,y]$ the geodesic segment from $x$ to $y$. If $x,y\in \partial X$, write $[x,y]$ for the geodesic path with Hopf coordinates $(x,y,t)_{t\in \mathbb{R}}$ (up to reparameterisation).
Let $L(\Gamma_0)$ be the limit set of $\Gamma_0$. Write $\Omega=p^{-1}(\Omega(\phi_0))$. We have that $\Omega$ is equivalently characterised as the set of vectors that whose Hopf coordinates are $(x,y,t)$, with $x,y\in L(\Gamma_0)$. 
In order to verify the (LVR) condition from Definition \ref{LVR} we use the following claim, whose proof is deferred until later.
\begin{claim}\label{geomclaim}
There is a constant $R>0$ such that for any $g\in \Gamma_0$, there is a geodesic $\gamma$ in $\Omega$ passing within distance $R$ of $x$ and of $g  x$, and moreover $\gamma$ does not intersect the boundary of any rectangle.
\end{claim}

Write $\mathcal{I}$ for the collection of $\phi$-orbit segments whose initial and terminal points lie in $\mathcal{S}$, and such that they do not intersect the boundary of any rectangle $gS^i$. We construct a map $\tau:\mathcal{I}\to \bigcup_n\mathcal{W}^n$. Let $\gamma\in \mathcal{I}$, and write its initial point as $z$. There is $n$ such that $\gamma$ is partitioned into orbit segments between $P^i z$ and $P^{i+1}z$, $i=0,\ldots, n-1$. Define $\tau(\gamma)=w\in \mathcal{W}^n$ by $P^i z\in g_iS^{w^i}$, for some (unique) $g_i$, for each $i=0,\ldots, n-1$. 

\begin{lemma}\label{LVR}
$(\Sigma, \psi, r)$ satisfies $\mathrm{(LVR)}$.
\end{lemma}

\begin{proof}
We construct a map $\lambda : \Gamma_0 \to \mathcal{S}$, and define $\chi: \Gamma_0\to \bigcup_n\mathcal{W}^n$ by $\chi=\tau\circ\lambda$.
Fix some $x\in X$. For each $T>0$, write $\mathcal{R}(T)\subset \Gamma_0$ for the elements $h\in \Gamma_0$ such that $\pi(hS^i)$ has distance at most $T$ to $x$, for some $i$. Notice that this set is necessarily finite.

Let $g\in \Gamma_0$ be arbitrary. Let $\gamma$ be the geodesic given in Claim \ref{geomclaim}.
Let $y^1,y^2\in SX$ be tangent to $\gamma$ with $d(\pi(y^1),x),d(\pi(y^2),gx)\le R$. Let $z^1,z^2\in \mathcal{S}$ with $\phi^s z^1= y^1$ and $\phi^t z^2= y^2$ for some $0\le s,t \le \epsilon$. It follows that $d_X(\pi(z^1),x),d_X(\pi(z^2),gx)\le R+2\epsilon$.
Define $\lambda(g) = [\pi(z^1),\pi(z^2)]$.
We write $w_g$ for $\chi(g)=\tau(\lambda(g))$ and $k_g$ for the length of $w_g$.
There are (unique) $h_1,h_2\in \mathcal{R}(R+2\epsilon)$ and $j_1,j_2$, such that $z^1\in h_1S^{j_1}$ and $z^2\in gh_2S^{j_2}$. 
Moreover, from the definition of $\psi$, we have that $gh_2 = h_1\psi^{k_g-1}(w_g)$.
Therefore $\psi^{k_g-1}(w_g)=h_1^{-1}gh_2$, verifying the first part of (LVR).

It remains to show the `linear' part of (LVR).
First, note that the length $|\lambda(g)|$ of the orbit segement $\lambda(g)$ satisfies
\[
d(x,gx) - 2R - 2\epsilon \le |\lambda(g)|\le d(x,gx) + 2R + 2\epsilon .
\]
By the semi-conjugacy with $\phi_0^t$, we have that $r^{k_g-1}(w_gz_g) = |\lambda(g)|$ for some $z_g\in \Sigma$.

Let $g_1,\cdots , g_m\in \Gamma_0$ be arbitrary.
Write $h_i=g_1\cdots g_i$, $h_0=e$, for each $i=1,\ldots,m$. We have
\[
|\lambda(h_m)|\le 2R + 2\epsilon+ \sum_{i=1}^m d(h_{i-1}x,h_ix) \le (m+1)(2R + 2\epsilon) + \sum_{i=1}^m |\lambda(g_i)|  ,
\]
and so
\[
\min(r)(k_{h_m}-1)\le (m+1)(2R + 2\epsilon) + \max(r) \sum_{i=1}^m k_{g_i}  \le (\max(r)+2R + 2\epsilon+1) \sum_{i=1}^m k_{g_i} ,
\]
as required.
\end{proof}

We now prove the claim.
\begin{repclaim}{geomclaim}
There is a constant $R>0$ such that for any $g\in \Gamma_0$, there is a geodesic $\gamma$ in $\Omega$ passing within distance $R$ of $x$ and of $g  x$, and moreover $\gamma$ does not intersect the boundary of any rectangle.
\end{repclaim}

\begin{proof}
We make use of the following material from the Patterson-Sullivan theory for $X$ and $\Gamma_0$. Our account is based on \cite{PPS}.
Since $\delta_{\Gamma_0}<\infty$, there exists a Patterson-Sullivan family  $\left\{\mu_x\right\}_{x\in X}$ whose support is precisely the limit set $L(\Gamma_0)$ (and has dimension $\delta_{\Gamma_0}$).
For a subset $A\subset X$, and a point $x\in X\cup \partial X$, define $\mathcal{O}_x A\subset \partial X$, \emph{the shadow of $A$ seen from $x$} to consist of end-points of geodesic rays (if $x\in X$) or lines (if $x\in\partial X$) starting from $x$ and meeting $A$.
We also use the notation $\times_\Delta$ to denote the direct product without the diagonal.

From Mohsen's shadow lemma \cite{PPS}, we conclude that there is $R^\prime$ with
\[
\mu_x\times \mu_{g  x} (\mathcal{O}_xB(g  x,R^\prime)\times \mathcal{O}_{g  x}B(x,R^\prime))>0.
\]
Since $\mu_x$ is supported on $L(\Gamma_0)$ we conclude that there is 
\[
(z_1,z_2)\in \left(\mathcal{O}_xB(g  x,R^\prime)\times_\Delta \mathcal{O}_{g  x}B(x,R^\prime)\right)\cap \left(L(\Gamma_0)\times_\Delta L(\Gamma_0)\right).
\] Moreover, since the set of $(y_1,y_2)\in \partial X \times_\Delta \partial X$ such that $[y_1,y_2]$ passes through the boundary of any rectangle $gS^i$ has zero $\mu_x\times \mu_{g  x}$-measure, we may assume that $(z_1,z_2)$ has been chosen such that $[z_1,z_2]$ does not pass through the boundary of any rectangle.
By the CAT($-\kappa$) property of $X$, there are $T_1,T_2>0$ (depending on $R^\prime$) such that $[z_1,z_2]$ passes within $T_1$ of $x$ and $g x$, provided $d(x,g  x)\ge T_1$. For a proof of this statement, see Lemma 3.17 of \cite{PPS}.
For those finitely many $g\in \Gamma_0$ with $d(x,g  x)< T_1$, we choose $D>0$ such that $D \ge d(gx, [z_1,z_2])+d(x, [z_1,z_2])$ (recall that $z_1,z_2$ are functions of $g$).
Thus taking $R=\max (D, T_2)$ completes the proof of the claim.
\end{proof}

It was stated in section \ref{ssft} that weak symmetry for the one-sided shift space is equivalent to a condition involving only periodic points. We include a proof here.
\begin{lemma}\label{weak}
A H\"{o}lder continuous $f:\Sigma^+\to\mathbb{R}$ is weakly symmetric with respect to $\dag$ if and only if there are $\beta(n)$ with $\beta(n)/n\to 0$ as $n\to \infty$ such that $|r^n(x)-r^n(x^\dag)|\le \beta(n)$ for any $x\in \Sigma^+$ with $\sigma^nx=x$, and $x^\dag\in\Sigma^+$ defined by $(x^\dag)^i = (x^{kn-i})^\dag$, for $i=0,\ldots, n-1$ and $k\in\mathbb{N}$.
\end{lemma}

\begin{proof}
Suppose that $f:\Sigma^+\to\mathbb{R}$ is weakly symmetric with respect to $\dag$. Let $x\in \Sigma^+$ with $\sigma^nx=x$, and write $w=x^0\ldots x^{n-1}\in\mathcal{W}^n$. Then $x\in [w]$ and $x^\dag\in[w^\dag]$ and so $|r^n(x)-r^n(x^\dag)|\le \log D_n$ by hypothesis.

For the converse, let $w\in\mathcal{W}^n$. By the aperiodicity of $\Sigma^+$, there is $u\in\mathcal{W}^p$ (where $p$ is the aperiodicity constant) such that $wuw$ is admissible. Define $x\in\Sigma^+$ by the infinite concatenation of $wu$. Then for any $y\in [w]$, and $z\in [w^\dag]$, we have $|r^n(y)-r^n(x)|,|r^n(z)-r^n(\sigma^p x^\dag)|\le |r|_\theta/(1-\theta)$. Therefore, 
\begin{align*}
&|r^n(y)-r^n(z)|
\\
&\le |r^n(z)-r^n(\sigma^px^\dag)| + |r^n(y)-r^n(x)| + |r^{n+p}(x)-r^{n+p}(x^\dag)| |r^{p}(\sigma^n x)-r^{p}(x^\dag)| 
\\
&\le C + \beta(n+p).
\end{align*}
For some constant $C>0$ independent of $w$ and $n$. The result follows by setting $D_n = e^Ce^{\beta(n+p)}$.
\end{proof}

It is notable that the function given in Proposition \ref{Bowen} and Lemma \ref{DS} concern the two-sided shift space, whereas Theorem \ref{transferop} is for the one-sided shift space. We can relate these in the following lemma. We say that a function, $f:\Sigma\to\mathbb{R}$, \emph{depends only on future coordinates} if $f(x)=f(y)$ when $x^i=y^i$ for all $i\in \mathbb{Z}_{\ge 0}$. In this way, we may consider $f$ to be a function $f : \Sigma^+\to \mathbb{R}$.
\begin{lemma}\label{liv}
For any H\"{o}lder continuous $r:\Sigma\to \mathbb{R}$, there is a H\"{o}lder continuous $r^\prime:\Sigma\to \mathbb{R}$ depending only on future coordinates, satisfying 
\[
\sum_{i=0}^{n-1}r^\prime(\sigma^i x) = 
\sum_{i=0}^{n-1}r(\sigma^i x) 
\]
for any $x\in\Sigma$ with $\sigma^nx=x$.
Moreover, if $(r,\psi)$ is weakly symmetric, then $(r^\prime,\psi)$ is weakly symmetric, and we have $P_{\mathrm{Gur}}(-sr^\prime,T_{\psi_\Gamma})
=P_{\mathrm{Gur}}(-sr,T_{\psi_\Gamma})$.
\end{lemma}
\begin{proof}
First, we can find a H\"{o}lder continuous $r^\prime:\Sigma\to \mathbb{R}$ depending on future coordinates which is cohomologous to $r$. Then, we note that since $r$ and $r^\prime$ are cohomologous, it follows that 
\[
\sum_{i=0}^{n-1}r^\prime(\sigma^i x) = 
\sum_{i=0}^{n-1}r(\sigma^i x) 
\]
for any $x\in\Sigma$ with $\sigma^nx=x$.
A good reference for this material is \cite{PP}.

The second statement follows easily from the first, as only periodic points appear in definition of the Gurevi\v{c} pressure, and periodic points are sufficient to verify weak symmetry by Lemma \ref{weak}. 
\end{proof}

\begin{remark}
Note that in the lemma above, $r^\prime$ may have a different H\"{o}lder exponent to $r$.
\end{remark}

We are now ready to prove the main theorem.
\begin{proof}[Proof of Theorem \ref{manifold}]


We begin with a proof of the following:
\[
\inf_{\Gamma\in\mathcal{N}}\kappa_A(\pi_{\Gamma_0/\Gamma},\mathds{1})> 0 \implies \sup_{\Gamma\in\mathcal{N}}\delta_\Gamma < \delta_{\Gamma_0}.
\]
Let $\mathcal{N}$ be an arbitrary collection of normal subgroups of $\Gamma_0$ with $\inf_{\Gamma\in \mathcal{N}} \kappa_A(\pi_{\Gamma_0/\Gamma},\mathds{1})>0$. Write $\kappa=\inf_{\Gamma\in \mathcal{N}} \kappa_A(\pi_{\Gamma_0/\Gamma},\mathds{1})$. Since the trivial group $\left\{1\right\}$ has $\delta_{\Gamma_0}>\delta_{\left\{1\right\}}$ by construction, we will assume from now that $\left\{1\right\}\notin \mathcal{N}$.

By Lemma \ref{DS}, in order to get a uniform bound on $\delta_\Gamma$, we just need to give a uniform bound for $s=h_\Gamma$ such that $P_{\mathrm{Gur}}(T_{\psi_\Gamma},-sr)=0$. Note that the unique value $s=h_0$ for which $P(-sr,\sigma)=0$ satisfies $h_0=h_{\Gamma_0}=\delta_{\Gamma_0}$

Though $\psi$ depends on two letters as opposed to one, we may still apply the conclusion of Theorem \ref{transferop}(\ref{transferop3}). That is, there are $\epsilon_1,\epsilon_2 >0$ such that for all $s\in [h_0-\epsilon_1,h_0]$ and all $\Gamma\in\mathcal{N}$,
\[
\text{spr}(\mathcal{L}_{-sr^\prime,\Gamma})
\le (1-\epsilon_1)\text{spr}(L_{-h_0r^\prime})=(1-\epsilon_2),
\]
noting that $\text{spr}(L_{-h_0r^\prime})=1$.
Moreover, since $T_{\psi_\Gamma}$ is transitive when $\Gamma\ne \left\{1\right\}$, we have that for all $s\in [h_0-\epsilon_2,h_0]$,
\[
P_{\mathrm{Gur}}(-sr^\prime,T_{\psi_\Gamma}) \le \log\text{spr}(\mathcal{L}_{-sr^\prime,\Gamma}).
\]
Hence for all $s\in [h_0-\epsilon_2,h_0]$
\[
P_{\mathrm{Gur}}(-sr,T_{\psi_\Gamma})
=P_{\mathrm{Gur}}(-sr^\prime,T_{\psi_\Gamma})
\le \log (1-\epsilon_1)<0,
\]
and so $h_\Gamma\le h_0-\epsilon_1$ as required.

We now proceed with the second part, completing the proof of Theorem \ref{manifold}. That is, we will prove that
\[
\inf_{\Gamma\in\mathcal{N}}\kappa_{A/\Gamma}(\pi_{\Gamma_0/\Gamma},\mathds{1})= 0 \implies \sup_{\Gamma\in\mathcal{N}}\delta_{\Gamma} = \delta_{\Gamma_0}.
\]

By Theorem \ref{transferop}(\ref{transferop1}),
\[
\sup_{\Gamma\in\mathcal{N}}P_{\mathrm{Gur}}(-sr^\prime,T_{\psi_\Gamma})= P(-sr^\prime,\sigma),
\]
for every $s$. In particular, for $s=h_0$,
\[
\sup_{\Gamma\in\mathcal{N}}P_{\mathrm{Gur}}(-h_0r^\prime,T_{\psi_\Gamma})= 0,
\]
and for every $\epsilon>0$,
\[
\sup_{\Gamma\in\mathcal{N}}P_{\mathrm{Gur}}(-(h_0-\epsilon)r^\prime,T_{\psi_\Gamma})= P(-(h_0-\epsilon)r,\sigma)<0.
\]
It follows that we can find a sequence $\Gamma_n$ with
\[
P_{\mathrm{Gur}}(-(h_0-\frac{1}{n})r,T_{\psi_{\Gamma_{n}}})< 0.
\]
Since $h_0-1/n\le h_{\Gamma_{n}}\le h_0$, we conclude that $h_{\Gamma_{n}}\to h_0$ as $n\to \infty$; and by Lemma \ref{DS}, $\delta_{\Gamma_{n}}\to \delta_{\Gamma_0}$ as $n\to\infty$.
\end{proof}

\section{Proof of Theorem \ref{transferop}(\ref{transferop1})}
We now return to the setting of subshifts of finite type and their group extensions. Let $\sigma : \Sigma^+ \to \Sigma^+$ be a mixing subshift of finite type and $T_\psi: \Sigma^+\times G\to \Sigma^+\times G$. Fix a finite generating set $A$ of $G$ and let $\mathcal{N}$ be a collection of normal subgroups of $G$.

The aim of this section is to prove the following theorem.
\begin{reptheorem}{transferop}[i]
Assume that $(\psi,r)$ is weakly symmetric. Then
\[
\inf_{H\in \mathcal{N}} \kappa_{A/H}(\pi_{G/H},\mathds{1}) = 0 \implies \sup_{H\in \mathcal{N}} P_{\mathrm{Gur}}(r,T_{\psi_{H}}) = P(r,\sigma).
\]
\end{reptheorem}

Write $\rho_H$ for the representation of $G$ in  $\mathcal{U}(\ell^2(G/H))$ induced by the action of $G$ on the cosets $G/H$. We have that $\kappa_{A}(\rho_{H},\mathds{1})=\kappa_{A/H}(\pi_{G/H},\mathds{1})$.

We make use of an argument found in \cite{Ollivier} which characterises property (T) in terms of the spectra of $G$-equivariant symmetric random walks. As we use a particular family of representations, we simplify the result to our setting.
Let $\mu: G \to [0,1]$ be a discrete probability measure with $\mu(g)=\mu(g^{-1})$ for all $g\in G$. In our setting, we always assume that $\mu$ has finite support. Define the random walk operator $M: \ell^2(G)\to \ell^2(G)$ by $Mf(x)=\sum_{g\in G}\mu(g)f(xg)$. In this way we can write $M=\sum_{g\in G} \mu(g)\pi_G(g)$. The operator $M$ descends to the quotients of $G$ in a straightforward way: for $H\unlhd G$ define $M_H : \ell^2(G/H)\to \ell^2(G/H)$ by $M_H = \sum_{g\in G} \mu(g)\rho_H(g)$. We write $\text{spr} (M_H)$ for the spectral radius of the operator $M_H$ on the space $\ell^2(G/H)$.

\begin{proposition}[Ollivier\cite{Ollivier}\label{Ollivier}]
Let $B=\mathrm{supp}(\mu)$. Then
\[
\sup_{H\in\mathcal{N}}\mathrm{spr}(M_H)<1 \implies
\inf_{H\in\mathcal{N}}\kappa_B(\rho_H,\mathds{1})>0.
\]

\end{proposition}


We present the short proof of this fact.

\begin{proof}
Write $\sigma = \sup_{H\in \mathcal{N}} \text{spr}(M_H)$.
Suppose that $v\in \ell^2(G/H)$ is $\epsilon$,$B$-invariant; that is,
\[
\left|\rho_H(b)v-v\right|\le \epsilon \left|v\right|,
\]
for all $b\in B$.
Then $\left|M_H v - v\right| \le \epsilon$. Expanding the norm, and noting the self-adjointness of $M_H$ we have $2\langle v,v \rangle - 2\langle M_H v, v \rangle \le \epsilon^2$. Rearranging gives that $\langle M_Hv,v\rangle \ge \langle v,v \rangle - \epsilon^2/2 = 1-\epsilon^2/2$. Since $\text{spr}(M_H) = \sup_{f\in \ell^2(G/H), \left|f\right|=1} \langle M_Hf, f\rangle$ it follows that $\text{spr}(M_H)\ge 1-\epsilon^2/2$. Therefore $\epsilon^2/2\ge 1-\sigma$, and so $\kappa_A(\rho_H,\mathds{1}) \ge \sqrt{2(1-\sigma)}$. As $\sigma$ is independent of $H\in\mathcal{N}$, we conclude that $\inf_{H\in\mathcal{N}}\kappa_A(\rho_H,\mathds{1})>0$. 
\end{proof}

We are now ready to prove the theorem.
\begin{proof}[Proof of Theorem \ref{transferop}(i)]
Assume that $(\psi,r)$ is weakly symmetric.
Assume that 
\[
\inf_{H\in\mathcal{N}}\kappa_{A/H}(\pi_{G/H},\mathds{1}) = 0.
\]
Recall that $\kappa_{A/H}(\pi_{G/H},\mathds{1})=\kappa_{A}(\rho_{H},\mathds{1})$. 

We make use of the following notation. For $a\in \mathcal{W}$, $n\in\mathbb{N}$ and $g\in G$, write
\[
\mathcal{W}^n_{a,a^\dag}(g) = \left\{ u\in \mathcal{W}^n : u^0=a, ua^\dag \text{ is admissible}, \psi^n(u)=g\right\},
\]
and
\[
\mathcal{W}^n_{a,a^\dag}(g,g^{-1}) = \mathcal{W}^n_{a,a^\dag}(g)\cup\mathcal{W}^n_{a,a^\dag}(g^{-1}).
\]

Define $\mu^{(1)}_{n},\mu^{(2)}_{n}: G \to [0,1]$ by
\[
\mu^{(1)}_{n}(g) = \frac{\sum_{u_1\in\mathcal{W}^n_{a,a^\dag}(g,g^{-1})}e^{r^n(u_1x_1)}}{2\sum_{u\in\mathcal{W}^n_{a,a^\dag}}e^{r^n(ux_1)}},
\]
for some $x_1\in [a^\dag]$; and
\[
\mu^{(2)}_{n}(g) = \frac{\sum_{u_2\in\mathcal{W}^n_{a^\dag,a}(g,g^{-1})}e^{r^n(u_2x_2)}}{2\sum_{u\in\mathcal{W}^n_{a^\dag,a}}e^{r^n(ux_2)}},
\]
for some $x_2\in [a]$.

Define $\mu_{2n}=\mu^{(1)}_{n}\star\mu^{(2)}_{n}$, where $\star$ indicates the convolution \[
\mu_{2n}(g) = \sum_{g_1,g_2\in G : g_1g_2=g}\mu^{(1)}_{n}(g_1) \mu^{(2)}_{n}(g_2).
\] 
Define the symmetric random walk operators 
\[
M_{2n,H}: \ell^2(G/H)\to \ell^2(G/H)
\]
\[
M_{2n,H} = \sum_{g\in G} \mu_{2n}(g)\rho_{H}.
\]
The spectral radius can be found to be 
\[
\text{spr}(M_{2n,H}) = \limsup_{k\to\infty}\langle M_{2n,H}^k\mathds{1}_{e_{G/H}},\mathds{1}_{e_{G/H}}\rangle^{1/k},
\] 
where $\mathds{1}_{e_{G/H}}\in \ell^2(G/H)$ is the indicator function on the identity of $G/H$. Therefore,
\[
\log \text{spr}(M_{2n,H}) = \limsup_{k\to\infty} \frac{1}{k}\log \sum_{h\in H} (\mu_{2n})^{\star k}(h).
\]
We claim that there is a sequence $C_n>0$ with $\limsup_{n\to\infty}C_n^{1/n}\to 1$ such that 
\[
\frac{1}{k}\log \sum_{h\in H} \mu_{2n}^{\star k}(h) \le \frac{1}{k}\log A_{n,k}- \log B_{n} + \log C_n ;
\]
where
\[
A_{n,k} = 
\sum_{h\in H}
\sum_{u\in\mathcal{W}^{2nk}_{a,a}(h)}e^{r^{2nk}(ux_2)} ,
\]
and
\[
B_{n} = \sum_{u\in\mathcal{W}^n_{a,a^\dag}}e^{r^n(ux_1)}\sum_{u\in\mathcal{W}^n_{a,a^\dag}}e^{r^n(ux_1)} .
\]
Note that 
\[
\limsup_{k\to\infty}\frac{1}{k}\log A_{n,k} \le  
2nP_{\mathrm{Gur}}(r,T_{\psi_{H}});
\] 
and
\[
\lim_{n\to\infty}\frac{1}{n}\log B_{n} = 2P(r,\sigma), 
\]
since $P(r,\sigma)= \lim_{n\to\infty} \frac{1}{n}\log \sum_{w\in \mathcal{W}^n_{u,v}}e^{r^n(wx)} $
for any $u,v\in\mathcal{W}$, $x\in [v]$.

Assuming the claim (whose proof we give later), we have that
\[
\frac{1}{n} \text{spr}(M_H) \le 2P_{\mathrm{Gur}}(r,T_{\psi_{H}}) - \frac{1}{n}\log B_{n} + \frac{1}{n}\log C_n .
\]

Write $B_{2n}$ for the support of $\mu_{2n}$. Since $A$ is assumed to generate $G$, and since $B_{2n}$ is finite, we have that
\[
\inf_{H\in\mathcal{N}} \kappa_{A}(\rho_{H},\mathds{1})\implies \inf_{H\in\mathcal{N}} \kappa_{B_{2n}}(\rho_{H},\mathds{1})=0,
\]
for each $n$. Therefore $\sup_{H\in\mathcal{N}} \text{spr}(M_{2n, H})=1$. We can choose a sequence $H_n$ such that 
\[
0 = \limsup_{n\to\infty} \frac{1}{n}\log \text{spr}(M_{2n,H_n}).
\]
We conclude that
\[
0 = \limsup_{n\to\infty} \frac{1}{n} \text{spr}(M_{2n, H_n}) \le \limsup_{n\to\infty} \left(2P_{\mathrm{Gur}}(r,T_{\psi_{H_n}}) - 2P(r,\sigma) \right) ,
\]
i.e. $\sup_n P_{\mathrm{Gur}}(r,T_{\psi_{H_n}}) = P(r,\sigma)$, as required.

We now prove the claim.
We have that
\[
\sum_{u_1\in\mathcal{W}^n_{a,a^\dag}(g,g^{-1})}e^{r^n(u_1x_1)} = 
\sum_{u_1\in\mathcal{W}^n_{a,a^\dag}(g)}e^{r^n(u_1x_1)}+e^{r^n(u^\dag_1x_1)},
\]
and by weak symmetry,
\[
\sum_{u_1\in\mathcal{W}^n_{a,a^\dag}(g,g^{-1})}e^{r^n(u_1x_1)} \le 
\sum_{u_1\in\mathcal{W}^n_{a,a^\dag}(g)}2D_n e^{r^n(u_1x_1)}.
\]
Therefore, 
\begin{align*}
&\sum_{\substack{g_1,\ldots,g_{2l}\in G :\\ g_1\cdots g_{2k}=h}} 
\prod_{i=1}^k 
\sum_{u_1\in\mathcal{W}^n_{a,a^\dag}(g_i,g_i^{-1})}e^{r^n(u_1x_1)}
\sum_{u_2\in\mathcal{W}^n_{a^\dag,a}(g_{i+1},g_{i+1}^{-1})}e^{r^n(u_2x_2)}
\\
&\le
\sum_{\substack{g_1,\ldots,g_{2k}\in G :\\ g_1\cdots g_{2k}=h}} 
\prod_{i=1}^k  
\sum_{u_1\in\mathcal{W}^n_{a,a^\dag}(g_{i})}2D_n e^{r^n(u_1x_1)}
\sum_{u_2\in\mathcal{W}^n_{a^\dag,a}(g_{i+1})}2D_n e^{r^n(u_2x_2)}
\\
&=
(2D_n)^{2k}\sum_{\substack{g_1,\ldots,g_{2l}\in G :\\ g_1\cdots g_{2k}=h}} 
\prod_{i=1}^k  
\sum_{u_1\in\mathcal{W}^n_{a,a^\dag}(g_{i})} e^{r^n(u_1x_1)}
\sum_{u_2\in\mathcal{W}^n_{a^\dag,a}(g_{i+1})} e^{r^n(u_2x_2)}.
\end{align*}
Using the H\"{o}lder property of $r$, for each $i$ we have,
\begin{align*}
&\sum_{u_1\in\mathcal{W}^n_{a,a^\dag}(g_i)}
\sum_{u_2\in\mathcal{W}^n_{a^\dag,a}(g_{i+1})}
e^{r^n(u_1x_1)}e^{r^n(u_2x_2)}
\\
&\le
\exp\left({|r|_\theta\sum_{j=0}^n \theta^j}\right)
\sum_{u\in\mathcal{W}^n_{a,a}(g_ig_{i+1})}
e^{r^{2n}(ux_2)}
\\
&\le c_\theta \sum_{u\in\mathcal{W}^n_{a,a}(g_ig_{i+1})}
e^{r^{2n}(ux_2)},
\end{align*}
where $c_\theta = \exp\left({|r|_\theta\sum_{j=0}^\infty \theta^j}\right)$. We then bound the $k$-length product as
\[
\prod_{i=1}^k
c_\theta
\sum_{u\in\mathcal{W}^n_{a,a}(g_ig_{i+1})}
e^{r^{2n}(ux_2)}
\le
c_\theta^{2k}
\sum_{u\in\mathcal{W}^n_{a,a}(h)}
e^{r^{2n}(ux_2)}.
\]

Putting these bounds together gives
\begin{align*}
&\frac{1}{k}\log \sum_{h\in H}\mu_{2n}^{\star k}(h) 
\\
&= 
\frac{1}{k}\log \left(\sum_{h\in H}\sum_{\substack{g_1,\ldots,g_{2k}\in G :\\ g_1\cdots g_{2k}=h}} 
\prod_{i=1}^k 
\sum_{u_1\in\mathcal{W}^n_{a,a^\dag}(g_i,g_i^{-1})}e^{r^n(u_1x_1)}
\sum_{u_2\in\mathcal{W}^n_{a^\dag,a}(g_{i+1}g_{i+1}^{-1})}e^{r^n(u_2x_2)}\right)
\\
&-\log \sum_{u\in\mathcal{W}^n_{a,a^\dag}}e^{r^n(ux_1)}\sum_{u\in\mathcal{W}^n_{a,a^\dag}}e^{r^n(ux_1)}
\\
&\le
\frac{1}{k}
\log
\sum_{h\in H}
\sum_{u\in\mathcal{W}^{2nk}_{a,a}(h)}e^{r^{2nk}(ux_2)}
-\log B_n + 2\log 2D_n
+
2\log c_\theta.
\end{align*}
Writing $C_n = (2D_nc_\theta)^2$, concludes the proof of the claim. 
\end{proof}

\section{Auxiliary lemmas}\label{auxilliary}


Recall that $L_r$ has an isolated simple maximum eigenvalue at $e^{P(r,\sigma)}$ and strictly positive eigenfunction $h\in F_\theta$. We say that a function $r$ is \emph{normalised} if $L_r 1 =1$. Setting $\hat{r} = r - \log h + \log h\circ \sigma - P(r,\sigma)$, we have that $L_{\hat{r}}1=1$, and moreover the spectra are related by $\text{spr}(\mathcal{L}_{r,H})=e^{P(r,\sigma)}\text{spr}(\mathcal{L}_{\hat{r},H})$. To see this, observe that the following inequality is satisfied pointwise for every $H$,
\[
e^{nP(r)}\frac{\inf_x h(x)}{\sup_x h(x)}(\mathcal{L}_{\hat{r},H})^n
\le
(\mathcal{L}_{r,H})^n \le 
e^{nP(r)}\frac{\sup_x h(x)}{\inf_x h(x)}(\mathcal{L}_{\hat{r},H})^n.
\]
Therefore it suffices to prove Theorem \ref{transferop}(\ref{transferop2}) under the assumption that $r$ is normalised. 

We write $\mathcal{C}^c_H\subset \mathcal{C}^\infty_H$ for the cone of non-negative functions that are constant in the $\Sigma^+$-coordinate; that is, $f\in \mathcal{C}^c_H$ if $f(x,g)=f(z,g)\ge 0$ for any $x,z\in \Sigma^+$  and $g\in G/H$. Note that $\mathcal{L}_{r,H}$ does not preserve this cone.

As in the previous section, we write $\rho_H: G \to \mathcal{U}(\ell^2(G/H))$ for the permutation representation determined by $H\unlhd G$. Write $\ell^2_+(G/H)$ for the cone of non-negative functions in $\ell^2(G/H)$.

\begin{lemma}\label{norm}
There exists a constant $C$ such that, for any $H\unlhd G$,
\begin{align*}
&\|(\mathcal{L}_{r,H})^n\| = \sup \left\{ \|(\mathcal{L}_{r,H})^nf\|: f\in \mathcal{C}^c_H, \|f\| = 1\right\}
\\
&\le C\sup \left\{ \left\Vert
\sum_{\sigma^n y = z}e^{r^n(y)}\rho_H(\psi^n(y))f
\right\Vert_{\ell^2(G/H)} : f\in \ell^2_+(G/H), \left\Vert f\right\Vert=1, z\in\Sigma^+ \right\}
\end{align*}
\end{lemma}

\begin{proof}
The first equality is straightforward: for any $f\in \mathcal{C}^\infty_H$, define $\hat{f}(x,g)=\sup_{x\in\Sigma}|f(x,g)|$. Then $\hat{f}\in \mathcal{C}^c_H$ with $\|\hat{f}\|=\|f\|$; and we have $\|(\mathcal{L}_{r,H})^n \hat{f}\|\ge \|(\mathcal{L}_{r,H})^n f\|$.

We now show the second inequality. We have
\begin{align*}
\|(\mathcal{L}_{r,H})^n f\| = &\|\sum_{u\in\mathcal{W}}\mathds{1}_{[u]}(\mathcal{L}_{r,H})^n f\| \\
&\le 
 \#\mathcal{W}\max_{u\in \mathcal{W}}\sqrt{\sum_{g\in G} \sup_{x\in [u]} \left|(\mathcal{L}_{r,H})^nf(x,g)\right|^2}.
\end{align*}
Write $u=v$ for the letter attaining this maximum and fix $z\in[v]$. For any $x\in [v]$ and any $f\in \mathcal{C}^c_H$ we have 
\[
(\mathcal{L}_{r,H})^nf(x,g) = \sum_{\sigma^n y=x} e^{r^n(y)} f(z,\psi_H^n(y)g).
\]
In addition,
\[
\sup_{x\in [v]}|(\mathcal{L}_{r,H})^nf(x,g)| \le \exp\left({\frac{|r|_\theta}{1-\theta}}\right) \sum_{\sigma^n y=z} e^{r^n(y)} f(z,\psi_H^n(y)g) .
\]
Writing $\hat{f}(g)=f(z,g)$, we have that $\hat{f}(g)\in \ell^2_+(G/H)$, and 
\[
(\mathcal{L}_{r,H})^nf(x,g) = \sum_{\sigma^n y=x} e^{r^n(y)} \rho_H(\psi^n(y))\hat{f}(g).
\]
Therefore we have,
\begin{align*}
\|(\mathcal{L}_{r,H})^n f\| = 
&\le \#\mathcal{W}\exp\left({\frac{|r|_\theta}{1-\theta}}\right)\sqrt{\sum_{g\in G} \left(\sum_{\sigma^n y=z} e^{r^n(y)} \rho_H(\psi^n(y))\hat{f}(g) \right)^2}
\\
&= C\left| \sum_{\sigma^n y =z} e^{r^n(y)}\rho_H(\psi^n(y))\hat{f} \right|
\end{align*}
with $C=\#\mathcal{W}\exp\left({\frac{|r|_\theta}{1-\theta}}\right)$. This completes the proof.
\end{proof}

As we have related the spectrum of the group extended transfer operators to expressions involving representations of $G$ in $\ell^2(G/H)$, we will make use of some general results about these representations.
\begin{lemma}[F{\o}lner\cite{Folner}]\label{nonamen}
If $G$ is non-amenable then:
\[
\forall \epsilon>0 : \, \exists B\subset G, B \text{ finite }: \forall E\subset G, E \text{ finite } : \exists b\in B : \# E\cap E\cdot b \le \epsilon \# E .
\]
\end{lemma}
\begin{proof}
In \cite{Folner} it is shown that $G$ is amenable if and only if there exists $\epsilon_0$ such that for any finite collection $a_1,\ldots, a_n\in G$ there exists a finite $E\subset G$ with 
\[
\frac{1}{n}\sum_{i=1}^n \# (E\cap Ea_i) \ge \epsilon_0\# E.
\]
Negating these statements gives that $G$ is non-amenable if and only if for all $\epsilon$ there exists a finite collection $a_1,\ldots, a_n\in G$ 
such that for every finite $E\subset G$ we have
\[
\frac{1}{n}\sum_{i=1}^n \# (E\cap Ea_i) \le \epsilon\# E .
\]
And since
\[
\frac{1}{n}\sum_{i=1}^n \# (E\cap Ea_i) \le \epsilon\# E \implies \#(E\cap Ea_i)\le \epsilon\#E \text{ for at least one } a_i
\]
the lemma follows.
\end{proof}

Let $\kappa>0$ and $\rho: G\to \mathcal{U}(\mathcal{H})$ be such that $\kappa_A(\rho,\mathds{1})\ge\kappa/2$. Let $f\in \mathcal{H}$ with $\left| f\right| =1$. By the parallelogram law, there is $a\in A$ with 
\[
\left| \rho(a)f+f\right| = 2\sqrt{1-\left| \rho(a)f-f\right| ^2/4}\le 2\sqrt{1-\kappa^2/16}.
\]
We write $\kappa_1=\sqrt{1-\kappa^2/16}$.

\begin{lemma}\label{nonamen}
Let $\epsilon>0$ be arbitrary. There is a finite subset $B=B(\epsilon)\subset G$, such that the following holds. Assume that $\rho: G\to \mathcal{U}(\mathcal{H})$ with $\kappa_A(\rho,\mathds{1})=\kappa>0$, for some finite subset $A$ of $G$. For any $f\in (\mathcal{H},\left| \cdot\right| )$ and for each $n$ we may choose $E_n=E_n(f,\epsilon)\subset G$ with the following properties:
\begin{enumerate}
\item $\left| \sum_{g\in E_n} \rho(g)f\right|  \le 2^n(1-\kappa_1)^n\left| f\right| $;
\item $\# E_n \ge 2^n(1-\epsilon)^n$
\item For any $g\in E_n$, we have $g\in (A\cup B)^n$.
\end{enumerate}
\end{lemma}
\begin{remark}
Notice that though $E_n$ depends on $f$ and on $\mathcal{H}$, the subset $B$ does not depend on $f$ or $\mathcal{H}$.
\end{remark}

\begin{proof}
Let $\epsilon>0$ be arbitrary. Let $B=B(\epsilon)$ be the finite set given in Lemma \ref{nonamen}. We may assume that $e\in B$.
Let $f\in \mathcal{H}$ be arbitrary.
We proceed by induction.

\noindent\textbf{Base case} $n=1$. We may choose $a\in A$ such that
\[
\left| \rho(a)f + f\right| \ge \kappa\left| f\right|  \ge \frac{\kappa}{2}\left| f\right| ,
\]
and so
\[
\left| \rho(a)f + f\right|  \le 2(1-\kappa_1)\left| f\right| ,
\]
satisfying condition 1.
Then setting $E_1 = \left\{a,e \right\}$ completes the base case of the induction.

\noindent\textbf{Inductive step.} Assume the claim is true for $n$.
Set $f_n= \sum_{g\in E_n}\rho(g)f$.
Let $a\in A$ such that $\left| \rho(a)f_n - f_n \right|  \ge \kappa\left| f_n\right| $.
Let $b\in B$ such that 
\[
\# (E_n\cup E_na) \cap (E_n\cup E_na) \cdot b \le \epsilon \# E_na=\epsilon\#E_n.
\]
Notice that it follows that
\[
\# E_na\cap E_nab \le \epsilon \# (E_n\cup E_na) \le 2\epsilon\# E_n,
\]
and similarly
\[
\# E_n\cap E_nab \le 2\epsilon\# E_n.
\]
We have two cases to consider:

\noindent Case 1. $\left| \rho(b)\rho(a)f_n - \rho(a)f_n \right|  \ge \frac{\kappa}{2}\left| f_n\right| $.
The result follows easily setting $E_{n+1}=E_na\cup E_nab$.

\noindent Case 2. $\left| \rho(b)f_n - f_n \right|  \ge \frac{\kappa}{2}\left| f_n\right| $.
In this case
\begin{align*}
\left|  \rho(ba)f_n - f_n \right|  &= \left|  \rho(ba)f_n - \rho(a)f_n +\rho(a)f_n -f_n \right|  \\
&\ge
\left| \rho(a)f_n-f_n\right|  - \left| \rho(b)\rho(a)f_n - \rho(a)f_n\right|  \\
&\ge \frac{\kappa}{2}\left| f_n\right| 
\end{align*}
The result follows by setting $E_{n+1} = E_n\cup E_nab$.
\end{proof}

\section{Proof of Theorem \ref{transferop}(ii)}
We are now almost in a position to prove the theorem. Therefore, assume that $(\Sigma^+,\psi,G)$ satisfies (LVR), and write the associated map $\chi$, linear constant $L$ and remainder set $\mathcal{R}\subset G$. Let $\mathcal{N}$ be a collection of normal subgroups of $G$ for which
\[
\inf_{H\in \mathcal{N}} \kappa_{A/H}(\pi_{G/H},\mathds{1}) = \kappa>0 ,
\]
for a finite generating set $A\subset G$;
Recall from section \ref{auxilliary} that it suffices to prove the theorem under the condition that $r$ is normalised.

The aim of this section is to find $N_1,N_2\in \mathbb{N}$ and $\eta(\kappa)>0$ such that, for any $H\in \mathcal{N}$, for any $f\in \ell^2_+(G/H)$, $\left|f\right|=1$, and any $x\in\Sigma^+$,
\[
\left|\sum_{\sigma^{nN_1} y = x} e^{r^{nN_1}(y)}\rho_H(\psi^{nN_1}(y))f \right|  \le (1-\eta(\kappa))^{nN_2}.
\]
In this case, by the inequality given in Lemma \ref{norm},
\begin{align*}
&\|((\mathcal{L}^{N_1}_{H,r})^n\|\\
&\le C
\sup\left\{\left|\sum_{\sigma^{nN_1} y = x} e^{r^{nN_1}(y)}\rho_H(\psi^{nN_1}(y))f \right| : f\in \ell^2_+(G/H), \left|f\right| =1, x\in \Sigma^+ \right\} 
\\
&\le C(1-\eta(\kappa))^{nN_2},
\end{align*}
and so $\text{spr}(\mathcal{L}_{r,H})\le (1-\eta(\kappa))^{\frac{N_2}{N_1}}<1$, as required.

We proceed with this aim. We simplify the notation by identifying $\psi_H$ with $\rho_H \circ \psi$. In this way our aim is to find
\[
\left|\sum_{\sigma^{nN_1} y = x} e^{r^{nN_1}(y)}\psi_H^{nN_1}(y)f \right| \le (1-\eta(\kappa))^{nN_2}.
\]
Recall that $\kappa_{A}(\rho_H,\mathds{1})=\kappa_{A/H}(\pi_{G/H},\mathds{1})$ and so, by hypothesis,
\[
\inf_{H\in \mathcal{N}} \kappa_{A}(\rho_H,\mathds{1}) = \kappa .
\]

Once and for all, fix $\epsilon>0$ sufficiently small to satisfy the following inequality
\[
(1-\epsilon)> (1-\kappa_1),
\]
where $\kappa_1=\sqrt{(1-\kappa^2/16}$, as in the previous section.
Now that $\epsilon$ is fixed, the set $B=B(\epsilon)$ from Lemma \ref{nonamen} is fixed.

The following constants appear in the estimations: $\alpha=\min_{x\in\Sigma^+} e^{r(x)}$, $p$ is the aperiodicity constant for $\Sigma^+$, $W=\#\mathcal{W}, R=\#\mathcal{R}$, $L$ is the (LVR) linear constant, and we write $K=\max_{g\in  A\cup B \cup \mathcal{R}} k_g$, where $k_g$ denotes the length of the word $\chi(g)$.
With these definitions
we have $\psi^{k_g}(\chi(g))=r_0(g)gr_1(g)$ with $k_g\le mKL$, for some $r_0(g),r_1(g)\in \mathcal{R}$, for each $m$, and each $g\in (A\cup B)^{m}$.
Fix $m$ sufficiently large to satisfy the inequality 
\[
\frac{(1-\epsilon)^{m}}{(1-\kappa_1)^m} > \exp{\frac{|r|_\theta}{1-\theta}}(mKLRW)^2.
\]
It will be useful to write
\[
\kappa_2 = \frac{1}{mKL}(1-\epsilon)^m-\exp{\frac{|r|_\theta}{1-\theta}} (RW)^2mKL(1-\kappa_1)^m>0.
\]

\begin{lemma}\label{paths}
Let $H\in\mathcal{N}$ be arbitrary.
For every $f\in \ell^2_+(G/H)$, $\left|f\right|=1$, and every $v_0,v_1\in\mathcal{W}$ there exists $P_{v_0,v_1}\subset \mathcal{W}_{v_0}^{mKL+2p}$ such that
\[
\left|\sum_{w\in P_{v_0,v_1}} \psi_H^{mKL+2p}(w)f\right| \le 2^m((RW)^2mKL(1-\kappa_1)^m).
\]
\end{lemma}

\begin{proof}
We describe a procedure to sandwich an arbitrary word $w$ between any two letters $v_0,v_1$; and then apply this to $w_g=\chi(g)$.

Let $v_0,v_1\in\mathcal{W}$ be arbitrary. For each $v\in\mathcal{W}$ fix $u_0(v)\in\mathcal{W}^p$ with initial letter $(u_0(v))^0= v_0$, and such that $u_0(v)v$ is admissible. 
For each $0\le s\le mKL$, and each $v\in\mathcal{W}$, fix $u_1(v,s)\in\mathcal{W}^{mKL+p-s}$ with initial letter $(u_1(v,s))^0=v$ and such that $u_1(v,s)v_1$ is admissible.

Let $H\in\mathcal{N}$ be arbitrary, and let $f\in \ell^2_+(G/H)$.
Fix 
\[
E_m=E_m\left(\sum_{r\in \mathcal{R}} \rho_H(r)\sum_{\substack{j\in \mathcal{W},\\ 0<s\le mKL+p}}\psi_H^{mKL+p-s}(u_1(s,j)) f\right).
\]

Define $P_{v_0,v_1}$ to be the collection of words $p_g=u_1(w_g^0,k_g)w_gu_0(w_g^{k_g-1})$, $g\in E_m$ (recall that $w^0_g$, $w^{k_g-1}_g$ denote the initial and terminal letters of $w_g$ respectively). Then $P_{v_0,v_1}\subset\mathcal{W}^{mKL+2p}$.
Since $f\ge 0$, we can make the following estimations,
\begin{align*}
&\left|\sum_{w\in P_{v_0,v_1}} \psi_H^{mKL+2p}(w) f\right|
=
\left|
\sum_{g\in E_m} \psi_H^{p}(u_0(w^0_g))\psi_H^{k_g}(w_g)\psi_H^{mKL+p-k_g}(u_1(w_g^{k_g-1},k_g)) f
\right|
\\
&\le 
\left|
 \sum_{i\in \mathcal{W}}\psi_H^{p}(u_0(i))\sum_{g\in E_m}\rho_H(r_0(g))\rho_H(g)\rho_H(r_1(g))\sum_{\substack{j\in \mathcal{W},\\ 0<s\le mKL}}\psi_H^{mKL+p-s}(u_1(s,j)) f
\right|
\\
&\le 
W
\left|
\sum_{r_0\in \mathcal{R}}\rho_H(r_0)\sum_{g\in E_m}\rho_H(g)\left(\sum_{r_1\in \mathcal{R}} \rho_H(r_1)\sum_{\substack{j\in \mathcal{W},\\ 0<s\le mKL}}\psi_H^{mNq+p-s}(u_1(s,j)) f\right)
\right|
\\
&\le RW 2^m(1-\kappa_1)^m
\left|\sum_{r_1\in \mathcal{R}} \rho_H(r_1)
\sum_{\substack{j\in \mathcal{W},\\ 0<s\le mKL}}\psi_H^{mKL+p-s}(u_1(s,j)) f
\right|
\\
&\le (RW)^2mKL2^m(1-\kappa_1)^m \left| f\right|.
\end{align*}
\end{proof}

For every $x\in [v_0]$ we may extend $P_{v_0,v_1}$ to $P_{x,v_1}\subseteq \sigma^{-(mKL+2p)}x$.
We informally refer to elements of $P_{x,v_1}$ as \emph{paths} (from $v_1$ to $x$).
We are now ready to prove the theorem.

\begin{proof}[Proof of Theorem \ref{transferop}(\ref{transferop2})]
For simplicity, we write $M=mKL+2p$. In order to prove the theorem recall it suffices to show that,
\[
\left|\sum_{\sigma^{nM} y = x} e^{r^{nM}(y)}\psi_H^{nM}(y)f \right| \le (1-\alpha^{M}2^m\kappa_2)^{nm}
\]
for any $H\in\mathcal{N}$, any $f\in \ell^2_+(G/H)$ with $\left|f\right|=1$, and for any $x\in \Sigma^+$. Note that $\alpha,m,M,\kappa_2$ do not depend on $H$ or on $f$.

\noindent\textbf{Base case.}
Let $x\in\Sigma^+$ be given. Fix some $v_1\in\mathcal{W}$ and let $P_{x,v_1}$ be given by Lemma \ref{paths}.
We have,
\begin{align*}
&\left|\sum_{\sigma^{nM} y = x} e^{r^{nM}(y)}\psi_H^{nM}(y)f \right| \\
&\le
\left|\sum_{w\in\mathcal{W}^{M}_x} (e^{r^{M}(wx)}-\alpha^M\mathds{1}_{P_{x,v_1}}(w))\psi_H^{M}(w)f\right| +
\left|\sum_{w\in\mathcal{W}^{M}_x, w\in P_{x,v_1}} \alpha^M\psi_H^{M}(w)f\right|\\
&\le
\left(1-\alpha^{M}\#P_{x,v_1}\right)
 +
\alpha^{M}\left|\sum_{w\in\mathcal{W}^{M}_x, w\in P_{x,v_1}} \psi_H^{M}(w)f\right|\\
&\le
\left(1-\frac{1}{mKL}2^m(1-\epsilon)^m\alpha^{M}\right) +
\alpha^{M}(RW)^2mKL2^m(1-\kappa_1)^m \\
&=
1-\alpha^{M}2^m\left(\frac{1}{mKL}(1-\epsilon)^m- (RW)^2mKL(1-\kappa_1)^m\right)\\
&\le 1-\alpha^{M}2^m\kappa_2.
\end{align*}

\noindent\textbf{Inductive step.} Assume that for every $f\in \ell^2_+(G/H)$, and every $x\in\Sigma^+$,
\[
\left|\sum_{\sigma^{nM} y = x} e^{r^{nM}(y)}\psi_H^{nM}(y)f \right| \le (1-\alpha^{M}\kappa_2)^{nm}.
\]
We will show that for every $f\in \ell^2_+(G/H)$ and every $x\in\Sigma^+$,
\[
\left|\sum_{\sigma^{(n+1)M} y = x} e^{r^{(n+1)M}(y)}\psi_H^{(n+1)M}(y)f \right| \le (1-\alpha^{M}\kappa_2)^{(n+1)m}.
\]
With this aim, let $f\in \ell^2_+(G/H)$ and $x\in\Sigma^+$ be arbitrary.
Fix some $v_1\in\mathcal{W}$ and $\hat{x}\in[v_1]$. Let $P_{x,v_1}$ correspond to $E_m=E_m\left(\sum_{\sigma^{nM} y = \hat{x}}\psi_H^{nM}(y)f\right)$. We have
\begin{align*}
&\left|\sum_{\sigma^{(n+1)M} y = x} e^{r^{(n+1)M}(y)}\psi_H^{(n+1)M}(y)f \right|
\\
&=
\left|\sum_{\sigma^{M}x^\prime=x}\sum_{\sigma^{nM} y = x^\prime} e^{r^{M}(x^\prime)}e^{r^{nM}(y)}\psi_H^{mKL+2}(x^\prime)\psi_H^{nM}(y)f \right|
\\
&\le\left|
\sum_{\sigma^{M}x^\prime=x}\sum_{\sigma^{nM} y = x^\prime} (e^{r^{M}(x^\prime)}-\alpha^M\mathds{1}_{P_{x,v_1}}(x^\prime))e^{r^{nM}(y)}\psi_H^{M}(x^\prime)\psi_H^{nM}(y)f\right|
\\
&+
\left|\sum_{x^\prime\in P_{x,v_1}}\sum_{\sigma^{nM} y =x^\prime} \alpha^M e^{r^{nM}(y)}\psi_H^{M}(x^\prime)\psi_H^{nM}(y)f 
\right|.
\end{align*}
We estimate the second term as
\begin{align*}
&\left|\sum_{x^\prime\in P_{x,v_1}}\sum_{\sigma^{nM} y =x^\prime} \alpha^{M}e^{r^{nM}(y)}\psi_H^{M}(x^\prime)\psi_H^{nM}(y)f 
\right|
\\
&\le
\alpha^{M}\exp\left({\sum_{i=1}^{M}|r|_\theta\theta^{i}}\right)
\left|\sum_{x^\prime\in P_{x,v_1}}\psi_H^{M}(x^\prime)\sum_{\sigma^{nM} y =\hat{x}} e^{r^{nM}(y)}\psi_H^{nM}(y)f \right| 
\\
&\le
\alpha^{M}\exp\left({\frac{|r|_\theta}{1-\theta}}\right)(RW)^2mKL2^m(1-\kappa_1)^m
\left|\sum_{\sigma^{nM} y =\hat{x}} e^{r^{nM}(y)}\psi_H^{nM}(y)f 
\right| .
\end{align*}
Write $x_{\max}$ for the element in $\Sigma^+$ maximizing
\[
z\mapsto \left|\sum_{\sigma^{nM} y = z}e^{r^{nM}(z)}\psi_H^{nM}(y)f\right| .
\]
We have,
\begin{align*}
&\left|
\sum_{\sigma^{M}x^\prime=x}\sum_{\sigma^{nM} y = x^\prime} (e^{r^{M}(x^\prime)}-\alpha^M\mathds{1}_{P_{x,v_1}}(x^\prime))e^{r^{nM}(y)}\psi_H^{M}(x^\prime)\psi^{nM}(y)f\right|
\\
&\le
\sum_{\sigma^{M}x^\prime=x} (e^{r^{M}(x^\prime)}-\alpha^M\mathds{1}_{P_{x,v_1}}(x^\prime))
\left|\sum_{\sigma^{nM} y = x^\prime}e^{r^{nM}(y)}\psi_H^{nM}(y)f\right|
\\
&\le
\sum_{\sigma^{M}x^\prime=x} (e^{r^{M}(x^\prime)}-\alpha^M\mathds{1}_{P_{x,v_1}}(x^\prime))
\left|\sum_{\sigma^{nM} y = x_{\max}}e^{r^{nM}(y)}\psi_H^{nM}(y)f\right|
\\
&\le
(1-\alpha^{M}\frac{1}{mKL}2^m(1-\epsilon)^m)
\left|\sum_{\sigma^{nM} y = x_{\max}}e^{r^{nM}(y)}\psi_H^{nM}(y)f\right| .
\end{align*}
We therefore conclude the inductive step,
\begin{align*}
&\left|\sum_{\sigma^{(n+1)M} y = x} e^{r^{(n+1)M}(y)}\psi_H^{(n+1)M}(y)f \right|
\\
&\le
\left(
\alpha^{M}\exp\left({\frac{|r|_\theta}{1-\theta}}\right) (RW)^2mKL2^m(1-\kappa_1)^m + 1-\frac{1}{mKL}2^m(1-\epsilon)^m\alpha^{M}
\right)
\\
&\times\left|\sum_{\sigma^{nM} y =x_{\max}} e^{r^{nM}(y)}\psi_H^{nM}(y)f \right|
\\
&\le
(1-\alpha^{M}2^m\kappa_2)
\left|\sum_{\sigma^{nM} y =x_{\max}} e^{r^{nM}(y)}\psi_H^{nM}(y)f \right|.
\end{align*}
\end{proof}

\section{Proof Theorem \ref{transferop}(\ref{transferop3})}
As in section \ref{auxilliary}, we normalise a H\"{o}lder function $f$ by setting $\hat{f} = f - \log h + \log h\circ \sigma - P(f,\sigma)$, where $h$ is the eigenfunction (of maximal eigenvalue) for $L_f$. Let $s\mapsto r_s\in F_\theta$ be continuous (in the $\|\cdot\|_\theta$ topology) for $s\in [-1,1]$. Write $h_s$ for the maximal eigenfunction for $L_{r_s}$. Since the maximal eigenvalue is simple and isolated for all $s$, it follows that $s\mapsto h_s$ is continuous for $s\in [-1,1]$; see for instance \cite[Chapter 4]{Kato}. Write $\hat{r}_s = r_s - \log h_s + \log h_s\circ \sigma - P(r_s,\sigma)$. Then it follows that $s\mapsto\hat{r}_s \in F_\theta$ is also continuous for $s\in [-1,1]$.


\begin{proof}[Proof of Theorem \ref{transferop}(\ref{transferop3})]
Assume that $\inf_{H\in \mathcal{N}}\kappa_{A/H}(\pi_{G/H},\mathds{1})=\kappa>0$. 
As in the previous section, choose $\epsilon<\kappa_1$ and choose $m$ such that
\[
\frac{(1-\epsilon)^m}{(1-\kappa_1)^m}> (mKLRW)^2 \exp\left({\frac{|\hat{r}_0|_\theta}{1-\theta}}\right).
\]
Since $s\mapsto \hat{r}_s$ is continuous, we may choose $\delta>0$ such that
\[
\frac{(1-\epsilon)^m}{(1-\kappa)^m}> (mKLRW)^2 \exp\left({\frac{|\hat{r}_s|_\theta}{1-\theta}}\right),
\]
for all $s\in [-\delta,\delta]$.
Write 
\[
\beta=\min_{s\in [-\delta,\delta]} \min_{x\in\Sigma^+}\exp\left({\hat{r}_s(x)}\right).
\]
and
\[
\kappa_2 =\left(
\frac{1}{mKL}(1-\epsilon)^m-
\left(\max_{s\in [-\delta,\delta]}\exp{\frac{|\hat{r}_s|_\theta}{1-\theta}}\right)
 (RW)^2mKL(1-\kappa_1)^m
\right),
\]
Note that $\kappa_1$ depends only on $\kappa$; $W$ and $p$ depend only on $\Sigma^+$; and $R$ depends only on $\psi$. Therefore $\kappa_2$ is invariant under $H\in\mathcal{H}$ and $s\in[-\delta,\delta]$.

Following the proof of Theorem \ref{transferop}(\ref{transferop2}) we deduce that for each $s\in [-\delta,\delta]$,
\[
\text{spr}(\mathcal{L}_{{r_s},H})\le (1-\beta^{mKL+2p}2^m\kappa_2)\text{spr}(L_{r_s}),
\]
i.e.
\[
\sup_{s\in[-\delta,\delta]}\sup_{H\in\mathcal{N}}\text{spr}(\mathcal{L}_{{r_s},H})<\text{spr}(L_{r_0}).
\]
\end{proof}

\end{document}